\documentclass{scrartcl}
\usepackage{cite}
\usepackage{lmodern}
\usepackage{upgreek}
\usepackage[T1]{fontenc}
\usepackage{textcomp}
\usepackage{mathrsfs}
\usepackage{amsmath}
\usepackage{amssymb}
\usepackage[dvipdfm]{graphicx}
\makeatletter
\long\def\@makecaption#1#2{%
  \vskip\abovecaptionskip
  \sbox\@tempboxa{#1. #2}%
  \ifdim \wd\@tempboxa >\hsize
    #1. #2\par
  \else
    \global \@minipagefalse
    \hb@xt@\hsize{\hfil\box\@tempboxa\hfil}%
  \fi
  \vskip\belowcaptionskip}
\makeatother
\usepackage{pifont}
\usepackage{amsthm}
\usepackage{typearea}
\typearea{14}
\theoremstyle{definition}
\newtheorem{lemma}{Lemma}[section]
\newtheorem{theorem}{Theorem}[section]
\newtheorem{remark}{Remark}[section]
\newtheorem{definition}{Definition}[section]
\newtheorem{example}{Example}[section]
\numberwithin{equation}{section}

\newcommand{\divv}{\mathrm{d}}
\newcommand{\diff}{\,\divv}
\newcommand{\imnum}{\mathrm{i}\,}
\renewcommand{\Re}{\operatorname{Re}}
\renewcommand{\Im}{\operatorname{Im}}
\newcommand{\Si}{\operatorname{Si}}
\newcommand{\EInt}{\operatorname{E}_1}
\renewcommand{\pi}{\uppi\,\!}

\DeclareMathSymbol{\round}{\mathord}{letters}{"40}
\renewcommand{\partial}{\textrm{\Pisymbol{psy}{"B6}}}
\DeclareMathOperator{\Order}{O}
\DeclareMathOperator{\rme}{e}
\DeclareMathOperator{\arcsinh}{arcsinh}

\newcommand{\domD}{\mathscr{D}}

\newcommand{\LC}{\mathbf{L}}
\newcommand{\textSE}[1]{\text{\tiny{\rm{SE#1}}}}
\newcommand{\textDE}[1]{\text{\tiny{\rm{DE#1}}}}
\newcommand{\textSEg}{\text{\tiny{\rm{SE}}}}
\newcommand{\textDEg}{\text{\tiny{\rm{DE}}}}
\newcommand{\SEt}[1]{\psi_{\textSE{#1}}}
\newcommand{\DEt}[1]{\psi_{\textDE{#1}}}
\newcommand{\SEtInv}[1]{\SEt{#1}^{-1}}
\newcommand{\DEtInv}[1]{\DEt{#1}^{-1}}
\newcommand{\SEtDiv}[1]{\SEt{#1}'}
\newcommand{\DEtDiv}[1]{\DEt{#1}'}
\title{\textbf{Error Estimates with Explicit Constants for Sinc Quadrature and
Sinc Indefinite Integration over Infinite Intervals}\footnote{This work was supported by the Grant-in-Aid for Young Scientists (B) (No. 24760060), JSPS, Japan.}}

\author{Tomoaki Okayama \\
Graduate School of Economics, Hitotsubashi University\\
\texttt{tokayama@econ.hit-u.ac.jp}
}

\date{\today}

\begin{document}
\maketitle
\begin{abstract}
The Sinc quadrature and the Sinc indefinite integration
are approximation formulas
for definite integration and indefinite integration, respectively,
which can be applied on any interval
by using an appropriate variable transformation.
Their convergence rates have been analyzed for typical cases
including finite, semi-infinite, and infinite intervals.
In addition, for verified automatic integration,
more explicit error bounds that are computable
have been recently given on a finite interval.
In this paper, such explicit error bounds are given
in the remaining cases on semi-infinite and infinite intervals.
\end{abstract}
% Put keywords appropriate to your paper here, as shown
%\keywords{Sinc numerical methods, verified numerical integration,
%trapezoidal formula, DE formula}

% Put your AMS subject classifications into the argument of
% the following command.
%\AMSsubj{65D30, 65D32, 65G99}

%#!latex 2012scan-error-estim
\section{Introduction} \label{sec:introduction}

The ``Sinc quadrature'' is an approximation
formula for the integral over the whole real line, expressed as
\begin{equation}
\int_{-\infty}^{\infty}F(x)\diff x
\approx h\sum_{k=-M}^{N} F(kh),
\label{approx:Sinc-quad}
\end{equation}
where $M,\,N,\,h$ are selected appropriately depending on $n$.
This approximation is also called the (truncated) ``trapezoidal formula.''
It is well known that
the formula~\eqref{approx:Sinc-quad} can achieve quite a fast,
\emph{exponential} convergence.
Furthermore, its optimality is proved
under a certain class of
functions~\cite{stenger78:_optim,sugihara97:_optim}.
Here, there are two important conditions to be satisfied:
(i) the interval of integration is $(-\infty,\,\infty)$,
and (ii) $|F(x)|$ decays exponentially as $x\to\pm \infty$.
%
%Its accuracy is roughly explained by the following two steps:
%(i) the ``discretization'' step
%\begin{equation*}
%\int_{-\infty}^{\infty}F(x)\diff x
%\approx h\sum_{k=-\infty}^{\infty} F(kh)
%\end{equation*}
%yields quite an accurate result if $F$ is analytic
%on the strip domain $\domD_d=\{\zeta\in\mathbb{C}:|\Im\zeta|< d\}$,
%and (ii) the ``truncation'' step
%\begin{equation*}
%h\sum_{k=-\infty}^{\infty} F(kh)
%\approx h\sum_{k=-M}^{N} F(kh)
%\end{equation*}
%also yields quite an accurate result if $|F(x)|$
%decays quickly as $x\to \pm\infty$.
%It is well known that
%the (truncated)
%\emph{trapezoidal formula} 
%yields a result with remarkably high accuracy,
%and actually its optimality
%The formula~\eqref{approx:trapezoid}
%is also called the \emph{Sinc quadrature},
%because it can be derived from the \emph{Sinc approximation}
%and be regarded as one of Sinc numerical methods.
%
%
In other cases,
an appropriate variable transformation
$t=\psi(x)$
should be employed, i.e., the given integral should be transformed as
$\int_{a}^{b} f(t)\diff t=\int_{-\infty}^{\infty}f(\psi(x))\psi'(x)\diff x$,
so that the above two conditions are met.

In this regard,
Stenger~\cite{stenger93:_numer} considered the following four typical cases:
%given the concrete transformations for some cases:
\begin{enumerate}
 \item $(a,\,b)=(-\infty,\,\infty)$, and $|f(x)|$ decays algebraically
as $x\to \pm\infty$,\vspace*{-0.3\baselineskip}
 \item $(a,\,b)=(0,\,\infty)$, and $|f(x)|$ decays algebraically
as $x\to \infty$,\vspace*{-0.3\baselineskip}
 \item $(a,\,b)=(0,\,\infty)$, and $|f(x)|$ decays (already) exponentially
as $x\to \infty$,\vspace*{-0.3\baselineskip}
 \item The interval $(a,\,b)$ is finite.
\end{enumerate}
He then gave the concrete transformations to be employed in all cases:
\begin{align*}
\SEt{1}(t)&=\sinh t,\\
\SEt{2}(t)&=\mathrm{e}^t,\\
\SEt{3}(t)&=\arcsinh(\mathrm{e}^t),\\
\SEt{4}(t)&=\frac{b-a}{2}\tanh\left(\frac{t}{2}\right)+\frac{b+a}{2},
\end{align*}
which are called the ``Single-Exponential (SE) transformations.''
Takahasi--Mori~\cite{takahasi74:_doubl}
proposed the following improved transformations:
\begin{align*}
\DEt{1}(t)&=\sinh[(\pi/2)\sinh t],\\
\DEt{2}(t)&=\mathrm{e}^{(\pi/2)\sinh t},\\
\DEt{3$\dagger$}(t)&=\mathrm{e}^{t-\exp(-t)},\\
\DEt{4}(t)&=\frac{b-a}{2}\tanh\left(\frac{\pi}{2}\sinh t\right)+\frac{b+a}{2},
\end{align*}
which are called the ``Double-Exponential (DE) transformations.''
In addition,
in case~3, another DE transformation
\begin{equation*}
\DEt{3}(t)=\log(1+\rme^{(\pi/2)\sinh t})
\end{equation*}
was proposed~\cite{muhammad03:_doubl}
so that its inverse function can be explicitly written
with elementary functions
%as
%$\DEtInv{3$\dagger$}(t)=\arcsinh[(1/\pi)\log(\mathrm{e}^t - 1)]$
(whereas $\DEt{3$\dagger$}(t)$ cannot).

Error analyses of the Sinc quadrature
combined with $\SEt{1}(t)$,\,\ldots,\,$\SEt{4}(t)$
% the variable transformations above
are given~\cite{stenger93:_numer}
in the following form:
\begin{align}
 |\text{Error}|&\leq C \rme^{-\sqrt{2 \pi d \mu n}},\nonumber
\intertext{and for $\DEt{1}(t),\,\DEt{2}(t),\,\DEt{3}(t),\,\DEt{4}(t)$,
their error analyses have been given~\cite{tanaka09:_deint} as
}
 |\text{Error}|&\leq C \rme^{-2 \pi d n/\log(8 d n/\mu)},
\label{leq:DE-Sinc-quad-error}
\intertext{and for $\DEt{3$\dagger$}(t)$, also given~\cite{tanaka09:_deint} as}
 |\text{Error}|&\leq C \rme^{-2 \pi d n/\log(2\pi d n/\mu)},
\label{leq:DE3-Sinc-quad-error}
\end{align}
where $\mu$ indicates the decay rate of the integrand,
$d$ denotes the width of the domain in which
the transformed integrand is analytic,
and $C$ is a constant independent of $n$.
In view of the inequalities above,
the accuracy of the approximation
can be guaranteed if the constant $C$ is explicitly given
in a computable form.
In fact, the explicit form of $C$ was revealed
in case~4 (the interval is finite)~\cite{okayama09:_error},
and the result was used for
verified automatic integration~\cite{yamanaka10:_fast}.

The main objective of this study is
to reveal the explicit form of $C$'s
in the remaining cases: 1--3 (the interval is not finite),
which enables us to bound the errors by computable terms.
%Numerical experiments that confirm the results will be shown in this talk.
%The analyses are done
%in both cases: SE transformation and DE transformation.

As a second objective, this paper improves
the DE transformation in case~3.
Instead of $\DEt{3}(t)$ or $\DEt{3$\dagger$}(t)$,
\begin{equation*}
\DEt{3$\ddagger$}(t)=\log(1 + \rme^{\pi\sinh t})
\end{equation*}
is introduced in this paper, and it is shown that
the error is estimated as
\begin{equation}
 |\text{Error}|\leq C \rme^{-2 \pi d n/\log(4 d n/\mu)},
\label{leq:DE3ddagger-Sinc-quad-error}
\end{equation}
while clarifying the constant $C$.
The rate of~\eqref{leq:DE3ddagger-Sinc-quad-error}
is better than~\eqref{leq:DE-Sinc-quad-error}
and~\eqref{leq:DE3-Sinc-quad-error}.
Furthermore,
by using Sugihara's nonexistence theorem~\cite{sugihara97:_optim},
it can be shown that
$\DEt{3$\ddagger$}$ is the best
among the possible variable transformations in case~3
(although the point is not discussed in this paper
and left for another occasion).

In addition to the ``Sinc quadrature'' described above,
similar results can be given for
the ``Sinc indefinite integration''
for indefinite integrals $\int_a^{\xi}f(t)\diff t$,
which is also examined in this paper.

The remainder of this paper is organized as follows.
%In Section~\ref{sec:improve_DE},
%improvement of the DE transformation in the case 3
%is discussed, and the new DE transformation $\DEt{3$\ddagger$}$
%is derived.
%The optimality of $\DEt{3$\ddagger$}$ is also shown.
The main results
of this paper
are stated in Sections~\ref{sec:Sinc-quad-estimates}
and~\ref{sec:Sinc-indef-estimates};
new error estimates
for the Sinc quadrature are presented
in Section~\ref{sec:Sinc-quad-estimates},
and for the Sinc indefinite integration
in Section~\ref{sec:Sinc-indef-estimates}.
Numerical examples are shown in Section~\ref{sec:numer-exam}.
All proofs of the presented theorems
are given in Section~\ref{sec:proofs}.

%\input{improve_DE.tex}

%#!latex 2012scan-error-estim
\section{Error Estimates with Explicit Constants for the Sinc Quadrature}
\label{sec:Sinc-quad-estimates}

In this section,
after reviewing existing results,
new error estimates for the Sinc quadrature are stated.
%\begin{itemize}
% \item Sinc quadrature with the SE transformation
%in Section~\ref{subsec:main-SE-Sinc-Quad},
% \item Sinc quadrature with the DE transformation
%in Section~\ref{subsec:main-DE-Sinc-Quad}.
%% \item Sinc indefinite integration with the SE transformation
%%in Section~\ref{subsec:main-SE-Sinc-Indef},
%% \item Sinc indefinite integration with the DE transformation
%%in Section~\ref{subsec:main-DE-Sinc-Indef}.
%\end{itemize}
Let us introduce some notation here.
Let $\domD_d$ be a strip domain
defined by $\domD_d=\{\zeta\in\mathbb{C}:|\Im \zeta|< d\}$
for $d>0$.
% with $0<d<\pi/2$.
Furthermore, let
%$\domD_d^{-}=\domD_d \cap \{\zeta\in\mathbb{C}:\Re\zeta< 0\}$,
$\domD_d^{-}=\{\zeta\in\domD_d :\Re\zeta< 0\}$
and
%$\domD_d^{+}=\domD_d \cap \{\zeta\in\mathbb{C}:\Re\zeta\geq 0\}$.
$\domD_d^{+}=\{\zeta\in\domD_d :\Re\zeta\geq 0\}$.
In all the theorems presented in Sections~\ref{sec:Sinc-quad-estimates}
and~\ref{sec:Sinc-indef-estimates},
$d$ is assumed to be a positive constant with $d<\pi/2$.
% $d\in(0,\,\pi/2)$ is always assumed
%(and thus omitted).
%implicitly.
%and thus it is omitted
%Throughout this paper,
%the integrand after a variable transformation (say $x=\psi(t)$)
%is assumed to be analytic
%on this domain $\domD_d$.
%In other words, the given integrand $f$
%is assumed to be analytic on the translated domain
%\begin{equation*}
For a variable transformation $\psi$,
$\psi(\domD_d)$ denotes the image of $\domD_d$ by the map $\psi$,
%expressed as
i.e.,
$\psi(\domD_d)=\left\{z=\psi(\zeta): \zeta\in\domD_d \right\}$.
%\label{eq:image-of-domD}
%\end{equation*}
%
%Let $I_1=(-\infty,\,\infty)$, $I_2=I_3=(0,\,\infty)$, and
Let $I_1=(-\infty,\,\infty)$, $I_2=I_3=(0,\,\infty)$, and
let us define the following three functions:
\begin{align*}
E_1(z;\gamma)&=\frac{1}{(1+z^2)^{(\gamma+1)/2}},\\
E_2(z;\alpha,\beta)&=\frac{z^{\alpha-1}}{(1+z^2)^{(\alpha+\beta)/2}},\\
E_3(z;\alpha,\beta)&=\left(\frac{z}{1+z}\right)^{\alpha-1}\rme^{-\beta z}.
\end{align*}
We write $E_i(z;\gamma,\gamma)$ as $E_i(z;\gamma)$ for short.
%Recall that we use the expression~\eqref{eq:image-of-domD}
%throughout this paper.

\subsection{Existing and New Error Estimates for the Sinc Quadrature with the SE Transformation}
\label{subsec:main-SE-Sinc-Quad}

Existing error analyses for the Sinc quadrature
with $\SEt{1}$, $\SEt{2}$, and $\SEt{3}$ are written in the following form
(Theorems~\ref{thm:SE1-Sinc-Quad} and~\ref{thm:SE2-Sinc-Quad}).

\begin{theorem}[Stenger~{\cite[Theorem~4.2.6]{stenger93:_numer}}]
\label{thm:SE1-Sinc-Quad}
Assume that $f$ is analytic in $\SEt{1}(\domD_d)$, and
%for $d$ with $0 < d < \pi/2$, and
there exist positive constants $K$, $\alpha$, and $\beta$ such that
%for all $z\in\SEt{1}(\domD_d^{-})$ and
%$z\in \SEt{1}(\domD_d^{+})$
%there exist positive constants $K$ and $\mu$ such that
\begin{align}
 |f(z)|&\leq K |E_1(z;\alpha)|
%\frac{1}{|1+z^2|^{(\alpha+1)/2}}
\label{leq:Sinc-quad-case1-alpha}
\intertext{for all $z\in\SEt{1}(\domD_d^{-})$, and}
 |f(z)|&\leq K |E_1(z;\beta)|
%\frac{1}{|1+z^2|^{(\beta+1)/2}}
\label{leq:Sinc-quad-case1-beta}
\end{align}
for all $z\in\SEt{1}(\domD_d^{+})$.
Let $\mu=\min\{\alpha,\,\beta\}$,
let $h$ be defined as
\begin{equation}
h=\sqrt{\frac{2\pi d}{\mu n}},
\label{eq:Def-SE-h}
\end{equation}
and let $M$ and $N$ be defined as
\begin{equation}
\begin{cases}
M=n,\quad N=\lceil\alpha n/\beta\rceil
 & \,\,\,(\text{if}\,\,\,\mu = \alpha),\\
N=n,\quad M=\lceil\beta n/\alpha\rceil
 &  \,\,\,(\text{if}\,\,\,\mu = \beta).
\end{cases}
\label{eq:Def-SE-Sinc-MN}
\end{equation}
%respectively.
Then there exists a constant $C_1$, independent of $n$, such that
\begin{equation}
\left|
\int_{I_1}f(t)\diff t
- h\sum_{k=-M}^N f(\SEt{1}(kh))\SEtDiv{1}(kh)
\right|
\leq C_1\rme^{-\sqrt{2\pi d \mu n}}.
\label{leq:SE1-Sinc-Quad}
\end{equation}
%where
\end{theorem}
\begin{theorem}[Stenger~{\cite[Theorem~4.2.6]{stenger93:_numer}}]
\label{thm:SE2-Sinc-Quad}
The following is true for $i=2$, $3$.
Assume that $f$ is analytic in $\SEt{$i$}(\domD_d)$, and
%for $d$ with $0 < d < \pi/2$, and
there exist positive constants $K$, $\alpha$, and $\beta$ such that
\begin{equation}
 |f(z)|\leq K | E_i(z;\alpha,\beta) |
%\frac{|z|^{\alpha-1}}{|1+z^2|^{(\alpha+\beta)/2}}
\label{leq:Sinc-quad-case2-alpha-beta}
\end{equation}
for all $z\in\SEt{$i$}(\domD_d)$.
Let $\mu=\min\{\alpha,\,\beta\}$,
let $h$ be defined as~\eqref{eq:Def-SE-h},
and let $M$ and $N$ be defined as~\eqref{eq:Def-SE-Sinc-MN}.
%\begin{equation}
%\begin{cases}
%M=n,\quad N=\lceil\alpha n/\beta\rceil
% & \,\,\,(\text{if}\,\,\,\mu = \alpha),\\
%N=n,\quad M=\lceil\beta n/\alpha\rceil
% &  \,\,\,(\text{if}\,\,\,\mu = \beta),
%\end{cases}
%\label{eq:Def-SE-Sinc-MN}
%\end{equation}
%respectively.
Then there exists a constant $C_i$, independent of $n$, such that
\begin{equation}
\left|
\int_{I_i}f(t)\diff t
- h\sum_{k=-M}^N f(\SEt{$i$}(kh))\SEtDiv{$i$}(kh)
\right|
\leq C_i\rme^{-\sqrt{2\pi d \mu n}}.
\label{leq:SE2-Sinc-Quad}
\end{equation}
\end{theorem}
%\begin{theorem}[Stenger~{\cite[Theorem~4.2.6]{stenger93:_numer}}]
%\label{thm:SE3-Sinc-Quad}
%Assume that $f$ is analytic in $\SEt{3}(\domD_d)$
%for $d$ with $0 < d < \pi/2$, and
%there exist positive constants $K,\,\alpha$, and $\beta$ such that
%\begin{equation}
%|f(z)|\leq K \left|\frac{z}{1+z}\right|^{\alpha-1}|\rme^{-z}|^{\beta}
%\label{leq:Sinc-quad-case3-alpha-beta}
%\end{equation}
%for all $z\in\SEt{3}(\domD_d)$.
%Let $\mu=\min\{\alpha,\,\beta\}$,
%let $h$ be defined as~\eqref{eq:Def-SE-h},
%and let $M$ and $N$ be defined as~\eqref{eq:Def-SE-Sinc-MN}.
%Then there exists a constant $C$, independent of $n$, such that
%\begin{equation}
%\left|
%\int_{0}^{\infty}f(t)\diff t
%- h\sum_{k=-M}^N f(\SEt{3}(kh))\SEtDiv{3}(kh)
%\right|
%\leq C\rme^{-\sqrt{2\pi d \mu n}}.
%\label{leq:SE3-Sinc-Quad}
%\end{equation}
%\end{theorem}

This paper explicitly estimates the constant $C_i$'s
in~\eqref{leq:SE1-Sinc-Quad} and~\eqref{leq:SE2-Sinc-Quad}
%\eqref{leq:SE3-Sinc-Quad}
as follows.

\begin{theorem}
\label{thm:SE1-Sinc-Quad-Explicit}
Let the assumptions in Theorem~\ref{thm:SE1-Sinc-Quad} be fulfilled.
Furthermore, let $\nu=\max\{\alpha,\,\beta\}$.
Then the inequality~\eqref{leq:SE1-Sinc-Quad} holds with
\[
 C_1=\frac{2^{\nu+1}K}{\mu}
\left\{
 \frac{2}{(1-\rme^{-\sqrt{2\pi d \mu}})\{\cos d\}^{\nu}} + 1
\right\}.
\]
\end{theorem}
\begin{theorem}
\label{thm:SE2-Sinc-Quad-Explicit}
Let the assumptions in Theorem~\ref{thm:SE2-Sinc-Quad} be fulfilled.
% with $i=2$.
Then the inequality~\eqref{leq:SE2-Sinc-Quad} holds with
\begin{align*}
 C_2&=\frac{2K}{\mu}
\left\{
 \frac{2}{(1-\rme^{-\sqrt{2\pi d \mu}})\{\cos d\}^{(\alpha+\beta)/2}} + 1
\right\},\\
C_3&=\frac{2K}{\mu}
\left\{
 \frac{2^{1+(\beta/2)}c_{\alpha,d}}{(1-\rme^{-\sqrt{2\pi d \mu}})\{\cos d\}^{(\alpha+\beta)/2}} + 2^{(1-\alpha+|1-\alpha|)/2}
\right\},
\end{align*}
where $c_{\alpha,d}$ is a constant defined by
\begin{equation}
 c_{\alpha,d}=
\begin{cases}
\left\{2
\left(1 + \frac{1}{\cos d}\right)\right\}^{(1-\alpha)/2}
%\frac{\cos^2(d/2)}{\cos d}
%
 & (\text{if}\,\,\,0<\alpha < 1),\\
2^{(\alpha-1)/2}
 & (\text{if}\,\,\,1\leq \alpha).
\end{cases}
\label{eq:def-c_alpha_d}
\end{equation}
\end{theorem}

\subsection{Existing and New Error Estimates for the Sinc Quadrature with the DE Transformation}
\label{subsec:main-DE-Sinc-Quad}

Existing error analyses for the Sinc quadrature
with $\DEt{1}$, $\DEt{2}$, $\DEt{3}$, and $\DEt{3$\dagger$}$
are written in the following form
(Theorems~\ref{thm:DE1-Sinc-Quad} and~\ref{Thm:DEt3-Sinc-quad}).

\begin{theorem}[Tanaka et al.~{\cite[Theorem 3.1]{tanaka09:_deint}}]
\label{thm:DE1-Sinc-Quad}
The following is true for $i=1$, $2$, $3$.
Assume that $f$ is analytic in $\DEt{$i$}(\domD_d)$, and
%for $d$ with $0 < d < \pi/2$, and
there exist positive constants $K$ and $\mu$
(with $\mu\leq 1$ in case $i=3$)
such that
%\[
$
|f(z)|\leq K |E_i(z;\mu)|
$
%\frac{1}{|1+z^2|^{(\mu+1)/2}}
%\]
for all $z\in\DEt{$i$}(\domD_d)$.
Then there exists a constant $C$, independent of $n$, such that
\begin{equation*}
\left|
\int_{I_i}f(t)\diff t
- h\sum_{k=-n}^n f(\DEt{1}(kh))\DEtDiv{1}(kh)
\right|
\leq C\rme^{-2\pi d n/\log(8 d n/\mu)},
%\label{leq:DE1-Sinc-Quad}
\end{equation*}
where
\begin{equation}
h=\frac{\log(8 d n/\mu)}{n}.
\label{eq:Def-DE-h}
\end{equation}
\end{theorem}
%\begin{theorem}[Tanaka et al.~{\cite[Theorem 3.1]{tanaka09:_deint}}]
%\label{thm:DE2-Sinc-Quad}
%The following is true for $i=2,\,3$.
%Assume that $f$ is analytic in $\DEt{$i$}(\domD_d)$, and
%%for $d$ with $0 < d < \pi/2$, and
%there exist positive constants $K$ and $\mu$
%(with $\mu\leq 1$ in the case $i=3$) such that
%\begin{equation*}
% |f(z)|\leq K | E_i(z;\mu,\mu)|
%%\frac{|z|^{\mu-1}}{|1+z^2|^{\mu}}
%%\label{leq:Sinc-quad-case2-mu}
%\end{equation*}
%for all $z\in\DEt{$i$}(\domD_d)$.
%%
%%let $h$ be defined as~\eqref{eq:Def-DE-h},
%%and let $M$ and $N$ be defined as
%%\begin{equation}
%%\begin{cases}
%%M=n,\quad N=\lceil\alpha n/\beta\rceil
%% & \,\,\,(\text{if}\,\,\,\mu = \alpha),\\
%%N=n,\quad M=\lceil\beta n/\alpha\rceil
%% &  \,\,\,(\text{if}\,\,\,\mu = \beta),
%%\end{cases}
%%\label{eq:Def-DE-Sinc-MN}
%%\end{equation}
%%respectively.
%Then there exists a constant $C$, independent of $n$, such that
%\begin{equation*}
%\left|
%\int_{0}^{\infty}f(t)\diff t
%- h\sum_{k=-n}^n f(\DEt{$i$}(kh))\DEtDiv{$i$}(kh)
%\right|
%\leq C\rme^{-2\pi d n/\log(8 d n/\mu)},
%%\label{leq:DE2-Sinc-Quad}
%\end{equation*}
%where $h$ be defined as~\eqref{eq:Def-DE-h}.
%\end{theorem}
\begin{theorem}[Tanaka et al.~{\cite[Theorem 3.1]{tanaka09:_deint}}]
\label{Thm:DEt3-Sinc-quad}
Assume that $f$ is analytic in $\DEt{3$\dagger$}(\domD_d)$, and
%for $d$ with $0 < d < \pi/2$, and
there exist positive constants $K$ and $\mu$
with $\mu\leq 1$ such that
%\begin{equation*}
$|f(z)|\leq K | E_3(z;\mu)|$
%\left|\frac{z}{1+z}\right|^{\mu-1}|\rme^{-z}|^{\mu}
%\label{leq:Sinc-quad-case3-mu}
%\end{equation*}
for all $z\in\DEt{3$\dagger$}(\domD_d)$.
Then there exists a constant $C$, independent of $n$, such that
\begin{equation*}
\left|
\int_{I_3}f(t)\diff t
-h\sum_{k=-n}^n f(\DEt{3$\dagger$}(kh))\DEtDiv{3$\dagger$}(kh)
\right|
\leq C \rme^{-2\pi d n/\log(2 \pi d n/\mu)},
\end{equation*}
where $h=\log(2 \pi d n/\mu)/n$.
\end{theorem}
%\begin{theorem}[Tanaka et al.~{\cite[Theorem 3.1]{tanaka09:_deint}}]
%\label{Thm:DEt3-dagger-Sinc-quad}
%Assume that $f$ is analytic in $\DEt{3$\dagger$}(\domD_d)$
%for $d$ with $0 < d < \pi/2$, and
%there exist constants $K>0$ and $0<\mu\leq 1$ such that
%the inequality~\eqref{leq:Sinc-quad-case3-mu} holds
%for all $z\in\DEt{3$\dagger$}(\domD_d)$.
%Then there exists a constant $C$, independent of $n$, such that
%\begin{equation*}
%\left|
%\int_0^{\infty}f(t)\diff t
%-h\sum_{k=-n}^n f(\DEt{3$\dagger$}(kh))\DEtDiv{3$\dagger$}(kh)
%\right|
%\leq C \rme^{-2\pi d n/\log(8 d n/\mu)},
%\end{equation*}
%where $h=\log(8 d n/\mu)/n$.
%\end{theorem}
\begin{remark}
As for Theorem~\ref{thm:DE1-Sinc-Quad} with $i=3$
and Theorem~\ref{Thm:DEt3-Sinc-quad},
although the condition `$\mu\leq 1$' is not assumed (only $\mu>0$ is assumed)
in the original
paper~\cite{tanaka09:_deint},
that condition is necessary to avoid the case
where $|z/(1+z)|^{\mu-1}=\infty$
at $z=-1$ (see $E_3(z;\mu)$).
\end{remark}

%In addition, existing error analyses for the Sinc quadrature
%with $\DEt{3}$ and $\DEt{3$\dagger$}$ are written
%as Theorems~\ref{Thm:DEt3-Sinc-quad} and~\ref{Thm:DEt3-dagger-Sinc-quad}.

As for case 1 (Theorem~\ref{thm:DE1-Sinc-Quad} with $i=1$)
and case 2 (Theorem~\ref{thm:DE1-Sinc-Quad} with $i=2$),
this paper not only explicitly estimates the constant $C$'s,
%but also rewrites the inequality~\eqref{leq:Sinc-quad-case2-mu}
%in a more generalized form:~\eqref{leq:Sinc-quad-case2-alpha-beta},
but also generalizes the approximation formula
from $\sum_{k=-n}^n$ to $\sum_{k=-M}^N$ as stated below.
Here, $x_{\gamma}$ is defined for $\gamma>0$ by
\[
 x_{\gamma} =
\begin{cases}
\arcsinh\left(\frac{\sqrt{1+\sqrt{1-(2\pi\gamma)^2}}}{2 \pi \gamma}\right)
&(\text{if}\,\,\,0<\gamma<1/(2\pi)),\\
\arcsinh(1)
&(\text{if}\,\,\,1/(2\pi)\leq \gamma),
\end{cases}
\]
which is introduced to determine the region of $x$
where $\cosh(x)\rme^{\pm\pi\gamma\sinh x}$ is
monotone
(see Okayama et al.~\cite[Proposition 4.17]{okayama09:_error}).
%as in 
%this paper explicitly estimates the constant $C$
%in~\eqref{leq:DE1-Sinc-Quad} as follows.
%,~\eqref{leq:DE2-Sinc-Quad},
%\eqref{leq:DE3-Sinc-Quad} as follows.

\begin{theorem}
\label{thm:DE1-Sinc-Quad-explicit}
Assume that $f$ is analytic in $\DEt{1}(\domD_d)$, and
%for $d$ with $0 < d < \pi/2$, and
there exist positive constants $K$, $\alpha$, and $\beta$
such that~\eqref{leq:Sinc-quad-case1-alpha} holds
for all $z\in\DEt{1}(\domD_d^{-})$,
and~\eqref{leq:Sinc-quad-case1-beta} holds
for all $z\in\DEt{1}(\domD_d^{+})$.
Let $\mu=\min\{\alpha,\,\beta\}$,
let $\nu=\max\{\alpha,\,\beta\}$,
let $h$ be defined as~\eqref{eq:Def-DE-h},
and let $M$ and $N$ be defined as
\begin{equation}
\begin{cases}
M=n,\quad N=n - \lfloor\log(\beta/\alpha)/h\rfloor
 & \,\,\,(\text{if}\,\,\,\mu = \alpha),\\
N=n,\quad M=n-\lfloor\log(\alpha/\beta)/h\rfloor
 &  \,\,\,(\text{if}\,\,\,\mu = \beta).
\end{cases}
\label{eq:Def-DE-Sinc-MN}
\end{equation}
Furthermore,
let $n$ be taken sufficiently large so that
$n\geq (\nu \rme)/(8 d)$,
$M h \geq x_{\alpha/2}$, and $N h\geq x_{\beta/2}$ hold.
Then it holds that
% there exists a constant $C$, independent of $n$, such that
\begin{equation*}
\left|
\int_{I_1}f(t)\diff t
- h\sum_{k=-M}^N f(\DEt{1}(kh))\DEtDiv{1}(kh)
\right|
\leq C_1\rme^{-2\pi d n/\log(8 d n/\mu)},
%\label{leq:SE2-Sinc-Quad}
\end{equation*}
where $C_1$ is a constant independent of $n$, expressed as
\[
 C_1=\frac{2^{\nu+1}K}{\mu}
\left\{
 \frac{2}{(1-\rme^{-\pi\mu\rme/4})\{\cos(\frac{\pi}{2}\sin d)\}^{\nu}\cos d}
 + \rme^{\pi\nu/4}
\right\}.
\]
\end{theorem}
\begin{theorem}
\label{thm:DE2-Sinc-Quad-explicit}
Assume that $f$ is analytic in $\DEt{2}(\domD_d)$, and
%for $d$ with $0 < d < \pi/2$, and
there exist positive constants $K$, $\alpha$, and $\beta$
such that~\eqref{leq:Sinc-quad-case2-alpha-beta} holds with $i=2$
for all $z\in\DEt{2}(\domD_d)$.
Let $\mu=\min\{\alpha,\,\beta\}$,
let $\nu=\max\{\alpha,\,\beta\}$,
let $h$ be defined as~\eqref{eq:Def-DE-h},
and let $M$ and $N$ be defined as~\eqref{eq:Def-DE-Sinc-MN}.
%\begin{equation}
%\begin{cases}
%M=n,\quad N=n - \lfloor\log(\beta/\alpha)/h\rfloor
% & \,\,\,(\text{if}\,\,\,\mu = \alpha),\\
%N=n,\quad M=n-\lfloor\log(\alpha/\beta)/h\rfloor
% &  \,\,\,(\text{if}\,\,\,\mu = \beta),
%\end{cases}
%\label{eq:Def-DE-Sinc-MN}
%\end{equation}
%respectively.
Furthermore,
let $n$ be taken sufficiently large so that
$n\geq (\nu \rme)/(8 d)$,
$M h \geq x_{\alpha/2}$, and $N h \geq x_{\beta/2}$ hold.
Then it holds that
% there exists a constant $C$, independent of $n$, such that
\begin{equation*}
\left|
\int_{I_2}f(t)\diff t
- h\sum_{k=-M}^N f(\DEt{2}(kh))\DEtDiv{2}(kh)
\right|
\leq C_2\rme^{-2\pi d n/\log(8 d n/\mu)},
%\label{leq:SE2-Sinc-Quad}
\end{equation*}
where $C_2$ is a constant independent of $n$, expressed as
\[
 C_2=\frac{2K}{\mu}
\left\{
 \frac{2}{(1-\rme^{-\pi\mu\rme/4})\{\cos(\frac{\pi}{2}\sin d)\}^{(\alpha+\beta)/2}\cos d}
 + \rme^{\pi\nu/4}
\right\}.
\]
\end{theorem}

As for case 3 (Theorem~\ref{thm:DE1-Sinc-Quad} with $i=3$
and Theorem~\ref{Thm:DEt3-Sinc-quad}),
this paper employs the improved variable transformation
$\DEt{3$\ddagger$}$ as described in the introduction, and
gives the error estimates
in a form similar to Theorems~\ref{thm:DE1-Sinc-Quad-explicit}
and~\ref{thm:DE2-Sinc-Quad-explicit}.
\begin{theorem}
\label{thm:DE3-Sinc-Quad-explicit}
Assume that $f$ is analytic in $\DEt{3$\ddagger$}(\domD_d)$, and
%for $d$ with $0 < d < \pi/2$, and
there exist positive constants $K$, $\beta$,
and $\alpha$ with $\alpha\leq 1$
such that~\eqref{leq:Sinc-quad-case2-alpha-beta} holds with $i=3$
for all $z\in\DEt{3$\ddagger$}(\domD_d)$.
Let $\mu=\min\{\alpha,\,\beta\}$,
let $\nu=\max\{\alpha,\,\beta\}$,
let $h$ be defined as
\begin{equation}
h=\frac{\log(4 d n/\mu)}{n},
\label{eq:Def-DE-h-half}
\end{equation}
and let $M$ and $N$ be defined as~\eqref{eq:Def-DE-Sinc-MN}.
Furthermore,
let $n$ be taken sufficiently large so that
$n\geq (\nu \rme)/(4 d)$,
$M h \geq x_{\alpha}$, and $N h \geq x_{\beta}$ hold.
%\begin{equation}
%\begin{cases}
%M=n,\quad N=n - \lfloor\log(\beta/\alpha)/h\rfloor
% & \,\,\,(\text{if}\,\,\,\mu = \alpha),\\
%N=n,\quad M=n-\lfloor\log(\alpha/\beta)/h\rfloor
% &  \,\,\,(\text{if}\,\,\,\mu = \beta),
%\end{cases}
%\label{eq:Def-DE-Sinc-MN}
%\end{equation}
%respectively.
Then it holds that
% there exists a constant $C$, independent of $n$, such that
\begin{equation*}
\left|
\int_{I_3}f(t)\diff t
- h\sum_{k=-M}^N f(\DEt{3$\ddagger$}(kh))\DEtDiv{3$\ddagger$}(kh)
\right|
\leq C_{3\ddagger}\rme^{-2\pi d n/\log(4 d n/\mu)},
%\label{leq:SE2-Sinc-Quad}
\end{equation*}
where $C_{3\ddagger}$ is a constant independent of $n$, expressed as
\begin{align}
 C_{3\ddagger}&=\frac{2K}{\mu}
\left\{
 \frac{2(\tilde{c}_{d})^{1-\alpha}}{(1-\rme^{-\pi\mu\rme/2})\{\cos(\frac{\pi}{2}\sin d)\}^{\alpha+\beta}\cos d}
 + %\rme^{\pi(1-\alpha)/12}\rme^{\pi \nu/2}
\rme^{\pi(1-\alpha+6\nu)/12}
\right\},\nonumber
\intertext{and where $\tilde{c}_{d}$ is a constant
expressed by using $c_d=1+\{{1}/{\cos(\frac{\pi}{2}\sin d)}\}$ as}
 \tilde{c}_{d}&=
\frac{1+\log(1+c_d)}{\log(1+c_d)}c_d.\label{tilde-c-alpha-d}
% c_d&=
%\frac{2\cos^2(\frac{\pi}{4}\sin d)}{\cos(\frac{\pi}{2}\sin d)}
%1+\frac{1}{\cos(\frac{\pi}{2}\sin d)}.\nonumber
\end{align}
\end{theorem}

%,~\eqref{leq:DE2-Sinc-Quad},
%\eqref{leq:DE3-Sinc-Quad} as follows.

%\begin{theorem}
%Let the assumptions in Theorem~\ref{thm:DE3-Sinc-Quad} be fulfilled.
%Then the inequality~\eqref{leq:DE3-Sinc-Quad} holds with
%\[
% C=\frac{2^{\beta/2}c_{\alpha,d}K}{\mu}
%\left\{
% \frac{2}{(1-\rme^{-\sqrt{2\pi d \mu}})\{\cos d\}^{(\alpha+\beta)/2}} + 1
%\right\},
%\]
%where $c_{\alpha,d}$ is a constant defined as
%\[
% c_{\alpha,d}=
%\begin{cases}
%\left\{\frac{\cos^2(d/2)}{\cos d}\right\}^{(1-\alpha)/2}
% & (\text{if}\,\,0<\alpha \leq 1),\\
%2^{(\alpha-1)/2}
% & (\text{if}\,\,1<\alpha).
%\end{cases}
%\]
%\end{theorem}

%#!latex 2012scan-error-estim

\section{Error Estimates with Explicit Constants for the Sinc Indefinite Integration}
\label{sec:Sinc-indef-estimates}

Sinc indefinite integration
is an approximation formula for the indefinite integral~\cite{haber93:_two},
expressed as
\begin{equation}
 \int_{-\infty}^{\xi}F(x)\diff x
\approx\sum_{k=-M}^N F(kh)J(k,h)(\xi),\quad \xi\in\mathbb{R}.
\label{eq:Def-Sinc-Indef}
\end{equation}
Here, the basis function $J(k,h)$ is defined by
\begin{equation*}
\label{Def:func-J}
J(k,h)(x)
=h\left\{\frac{1}{2}+\frac{1}{\pi}\Si[\pi(x/h-k)]\right\},
\end{equation*}
where $\Si(x)$ is the so-called sine integral,
defined by
$\Si(x)=\int_0^x\{\sin(\sigma)/\sigma\}\diff\sigma$.
The approximation~\eqref{eq:Def-Sinc-Indef}
can be combined with
the SE transformation
or the DE transformation~\cite{muhammad03:_doubl}
similar to the Sinc quadrature~\eqref{approx:Sinc-quad}.
This section
% is devoted to
presents the error estimates for those formulas.

\subsection{New Error Estimates for the Sinc Indefinite Integration with the SE Transformation}
\label{subsec:main-SE-Sinc-Indef}

This paper gives new error estimates for the Sinc indefinite integration
with $\SEt{1}$, $\SEt{2}$, and $\SEt{3}$ in the following form
(Theorems~\ref{thm:SE1-Sinc-Indef} and \ref{thm:SE2-Sinc-Indef}).

\begin{theorem}
\label{thm:SE1-Sinc-Indef}
Assume that $f$ is analytic in $\SEt{1}(\domD_d)$, and
%for $d$ with $0 < d < \pi/2$, and
there exist positive constants $K$, $\alpha$, and $\beta$
such that~\eqref{leq:Sinc-quad-case1-alpha} holds
for all $z\in\SEt{1}(\domD_d^{-})$,
and~\eqref{leq:Sinc-quad-case1-beta} holds
for all $z\in\SEt{1}(\domD_d^{+})$.
Let $\mu=\min\{\alpha,\,\beta\}$,
let $\nu=\max\{\alpha,\,\beta\}$,
let $h$ be defined as
\begin{equation}
h=\sqrt{\frac{\pi d}{\mu n}},
\label{eq:Def-SE-h-indef}
\end{equation}
and let $M$ and $N$ be defined as~\eqref{eq:Def-SE-Sinc-MN}.
Then, it holds that
\begin{equation*}
\sup_{\tau\in I_1}\left|
\int_{-\infty}^{\tau}f(t)\diff t
- \sum_{k=-M}^N f(\SEt{1}(kh))\SEtDiv{1}(kh)J(k,h)(\SEtInv{1}(\tau))
\right|
\leq C_1\rme^{-\sqrt{\pi d \mu n}},
%\label{leq:SE1-Sinc-Quad}
\end{equation*}
where $C_1$ is a constant independent of $n$, expressed as
\begin{align*}
 C_1&=\frac{2^{\nu+1}K}{\mu}
\left\{
 \frac{1}{(1-\rme^{-2\sqrt{\pi d \mu}})\{\cos d\}^{\nu}}
 \sqrt{\frac{\pi}{d \mu}}
 + 1.1
\right\}.
\end{align*}
\end{theorem}
\begin{theorem}
\label{thm:SE2-Sinc-Indef}
The following is true for $i=2$, $3$.
Assume that $f$ is analytic in $\SEt{$i$}(\domD_d)$, and
%for $d$ with $0 < d < \pi/2$, and
there exist positive constants $K$, $\alpha$, and $\beta$
such that~\eqref{leq:Sinc-quad-case2-alpha-beta} holds
% with $i=2$
for all $z\in\SEt{$i$}(\domD_d)$.
Let $\mu=\min\{\alpha,\,\beta\}$,
let $h$ be defined as~\eqref{eq:Def-SE-h-indef},
and let $M$ and $N$ be defined as~\eqref{eq:Def-SE-Sinc-MN}.
Then, it holds that
\begin{equation*}
\sup_{\tau\in I_i}\left|
\int_{0}^{\tau}f(t)\diff t
- \sum_{k=-M}^N f(\SEt{$i$}(kh))\SEtDiv{$i$}(kh)J(k,h)(\SEtInv{$i$}(\tau))
\right|
\leq C_i\rme^{-\sqrt{\pi d \mu n}},
%\label{leq:SE2-Sinc-Quad}
\end{equation*}
where $C_2$ and $C_3$ are constants independent of $n$, expressed as
\begin{align*}
 C_2&=\frac{2K}{\mu}
\left\{
 \frac{1}{(1-\rme^{-2\sqrt{\pi d \mu}})\{\cos d\}^{(\alpha+\beta)/2}}
 \sqrt{\frac{\pi}{d \mu}}
 + 1.1
\right\},\\
 C_3&=\frac{2K}{\mu}
\left\{
 \frac{2^{1+(\beta/2)}c_{\alpha,d}}{(1-\rme^{-2\sqrt{\pi d \mu}})\{\cos d\}^{(\alpha+\beta)/2}}
 \sqrt{\frac{\pi}{d \mu}}
 + 1.1\cdot 2^{(1-\alpha+|1-\alpha|)/2}
\right\},
\end{align*}
and where $c_{\alpha,d}$ is a constant defined in~\eqref{eq:def-c_alpha_d}.
\end{theorem}
%\begin{theorem}
%\label{thm:SE3-Sinc-Indef}
%Assume that $f$ is analytic in $\SEt{3}(\domD_d)$
%for $d$ with $0 < d < \pi/2$, and
%there exist positive constants $K$, $\alpha$, and $\beta$
%such that~\eqref{leq:Sinc-quad-case2-alpha-beta} holds with $i=3$
%for all $z\in\SEt{3}(\domD_d)$.
%Let $\mu=\min\{\alpha,\,\beta\}$,
%let $h$ be defined as~\eqref{eq:Def-SE-h-indef},
%and let $M$ and $N$ be defined as~\eqref{eq:Def-SE-Sinc-MN}.
%Then it holds that
%\begin{equation*}
%\sup_{\tau\in [0,\,\infty)}\left|
%\int_{0}^{\tau}f(t)\diff t
%- \sum_{k=-M}^N f(\SEt{3}(kh))\SEtDiv{3}(kh)J(k,h)(\SEtInv{3}(\tau))
%\right|
%\leq C\rme^{-\sqrt{\pi d \mu n}},
%%\label{leq:SE2-Sinc-Quad}
%\end{equation*}
%where $C$ is a constant independent of $n$, expressed as
%\begin{align*}
% C&=\frac{2K}{\mu}
%\left\{
% \frac{2^{(\alpha+\beta-1)/2}c_{\alpha,d}}{(1-\rme^{-2\sqrt{\pi d \mu}})\{\cos d\}^{(\alpha+\beta)/2}}
% \sqrt{\frac{\pi}{d \mu}}
% + 1.1
%\right\},
%\end{align*}
%where $c_{\alpha,d}$ is a constant defined in~\eqref{eq:def-c_alpha_d}.
%\end{theorem}

\begin{remark}
This paper addresses
the indefinite integration formulas based on~\eqref{eq:Def-Sinc-Indef}
developed by Haber~\cite{haber93:_two}.
Haber developed his formula for case 4,
but did not develop any formula for cases~1-3.

Other indefinite integration formulas
with $\SEt{1}$, $\SEt{2}$, and $\SEt{3}$
were developed by Stenger~\cite{stenger93:_numer},
but error estimates of the formulas are left for future work.
%which contains two approximations: the Sinc approximation
%and the formula~\eqref{eq:Def-Sinc-Indef}.
%Therefore, the error estimate of the Sinc approximation
%is needed, which is a future work.
\end{remark}

\subsection{New Error Estimates for the Sinc Indefinite Integration with the DE Transformation}
\label{subsec:main-DE-Sinc-Indef}

This paper gives new error estimates for the Sinc indefinite integration
with $\DEt{1}$, $\DEt{2}$, and $\DEt{3$\ddagger$}$ in the following form
(Theorems~\ref{thm:DE1-Sinc-Indef}--\ref{thm:DE3-Sinc-Indef}).
Let us define $\epsilon^{\textDEg}_{d,\mu}(n)$ as
$\epsilon^{\textDEg}_{d,\mu}(n)=[\rme^{-\pi d n/\log(4 d n/\mu)}\log(4 d n/\mu)]/n$
for short.
\begin{theorem}
\label{thm:DE1-Sinc-Indef}
Assume that $f$ is analytic in $\DEt{1}(\domD_d)$, and
%for $d$ with $0 < d < \pi/2$, and
there exist positive constants $K$, $\alpha$, and $\beta$
such that~\eqref{leq:Sinc-quad-case1-alpha} holds
for all $z\in\DEt{1}(\domD_d^{-})$,
and~\eqref{leq:Sinc-quad-case1-beta} holds
for all $z\in\DEt{1}(\domD_d^{+})$.
Let $\mu=\min\{\alpha,\,\beta\}$,
let $\nu=\max\{\alpha,\,\beta\}$,
let $h$ be defined as~\eqref{eq:Def-DE-h-half},
and let $M$ and $N$ be defined as~\eqref{eq:Def-DE-Sinc-MN}.
Furthermore,
let $n$ be taken sufficiently large so that
$n\geq (\nu \rme)/(4 d)$,
$M h \geq x_{\alpha/2}$, and $N h\geq x_{\beta/2}$ hold.
Then, it holds that
% there exists a constant $C$, independent of $n$, such that
\begin{equation*}
\sup_{\tau\in I_1}\left|
\int_{-\infty}^{\tau}f(t)\diff t
- \sum_{k=-M}^N f(\DEt{1}(kh))\DEtDiv{1}(kh)J(k,h)(\DEtInv{1}(\tau))
\right|
\leq C_1 \epsilon^{\textDEg}_{d,\mu}(n),
%\rme^{-\pi d n/\log(4 d n/\mu)},
%\label{leq:SE2-Sinc-Quad}
\end{equation*}
where $C_1$ is a constant independent of $n$, expressed as
\[
 C_1=\frac{2^{\nu+1}K}{\mu d}
\left\{
 \frac{1}{(1-\rme^{-\pi\mu\rme/2})\{\cos(\frac{\pi}{2}\sin d)\}^{\nu}\cos d}
 + \rme^{\pi(\alpha+\beta)/4}
\right\}.
\]
\end{theorem}
\begin{theorem}
\label{thm:DE2-Sinc-Indef}
%The following true for $i=2,\,3$.
Assume that $f$ is analytic in $\DEt{2}(\domD_d)$, and
%for $d$ with $0 < d < \pi/2$, and
there exist positive constants $K$, $\alpha$, and $\beta$
such that~\eqref{leq:Sinc-quad-case2-alpha-beta} holds with $i=2$
for all $z\in\DEt{2}(\domD_d)$.
Let $\mu=\min\{\alpha,\,\beta\}$,
let $\nu=\max\{\alpha,\,\beta\}$,
let $h$ be defined as~\eqref{eq:Def-DE-h-half},
and let $M$ and $N$ be defined as~\eqref{eq:Def-DE-Sinc-MN}.
%\begin{equation}
%\begin{cases}
%M=n,\quad N=n - \lfloor\log(\beta/\alpha)/h\rfloor
% & \,\,\,(\text{if}\,\,\,\mu = \alpha),\\
%N=n,\quad M=n-\lfloor\log(\alpha/\beta)/h\rfloor
% &  \,\,\,(\text{if}\,\,\,\mu = \beta),
%\end{cases}
%\label{eq:Def-DE-Sinc-MN}
%\end{equation}
%respectively.
Furthermore,
let $n$ be taken sufficiently large so that
$n\geq (\nu \rme)/(4 d)$,
$M h \geq x_{\alpha/2}$, and $N h \geq x_{\beta/2}$ hold.
Then, it holds that
% there exists a constant $C$, independent of $n$, such that
\begin{equation*}
\sup_{\tau\in I_2}\left|
\int_{0}^{\tau}f(t)\diff t
- \sum_{k=-M}^N f(\DEt{2}(kh))\DEtDiv{2}(kh)J(k,h)(\DEtInv{2}(\tau))
\right|
\leq C_2\epsilon^{\textDEg}_{d,\mu}(n),
%\rme^{-2\pi d n/\log(4 d n/\mu)},
%\label{leq:SE2-Sinc-Quad}
\end{equation*}
where $C_2$ is a constant independent of $n$, expressed as
\begin{align*}
C_2&=\frac{2K}{\mu d}
\left\{
 \frac{1}{(1-\rme^{-\pi\mu\rme/2})\{\cos(\frac{\pi}{2}\sin d)\}^{(\alpha+\beta)/2}\cos d}
 + \rme^{\pi(\alpha+\beta)/4}
\right\}.
\end{align*}
%where $\tilde{c}_{d}$ is a constant defined in~\eqref{tilde-c-alpha-d}.
\end{theorem}
\begin{theorem}
\label{thm:DE3-Sinc-Indef}
Assume that $f$ is analytic in $\DEt{3$\ddagger$}(\domD_d)$, and
%for $d$ with $0 < d < \pi/2$, and
there exist positive constants $K$, $\beta$,
and $\alpha$ with $\alpha\leq 1$
such that~\eqref{leq:Sinc-quad-case2-alpha-beta} holds with $i=3$
for all $z\in\DEt{3$\ddagger$}(\domD_d)$.
Let $\mu=\min\{\alpha,\,\beta\}$,
let $\nu=\max\{\alpha,\,\beta\}$,
let $h$ be defined as $h=\log(2 d n/\mu)/n$,
and let $M$ and $N$ be defined as~\eqref{eq:Def-DE-Sinc-MN}.
Furthermore,
let $n$ be taken sufficiently large so that
$n\geq (\nu \rme)/(2 d)$,
$M h \geq x_{\alpha}$, and $N h \geq x_{\beta}$ hold.
%\begin{equation}
%\begin{cases}
%M=n,\quad N=n - \lfloor\log(\beta/\alpha)/h\rfloor
% & \,\,\,(\text{if}\,\,\,\mu = \alpha),\\
%N=n,\quad M=n-\lfloor\log(\alpha/\beta)/h\rfloor
% &  \,\,\,(\text{if}\,\,\,\mu = \beta),
%\end{cases}
%\label{eq:Def-DE-Sinc-MN}
%\end{equation}
%respectively.
Then, it holds that
% there exists a constant $C$, independent of $n$, such that
\begin{equation*}
\sup_{\tau\in I_3}\left|
\int_{0}^{\tau}f(t)\diff t
- \sum_{k=-M}^N f(\DEt{3$\ddagger$}(kh))\DEtDiv{3$\ddagger$}(kh)
  J(k,h)(\DEtInv{3$\ddagger$}(\tau))
\right|
\leq C_{3\ddagger}\epsilon^{\textDEg}_{d,2\mu}(n),
%\label{leq:SE2-Sinc-Quad}
\end{equation*}
where $C_{3\ddagger}$ is a constant independent of $n$, expressed as
\begin{align*}
 C_{3\ddagger}&=\frac{2K}{\mu d}
\left\{
 \frac{(\tilde{c}_{d})^{1-\alpha}}{(1-\rme^{-\pi\mu\rme})\{\cos(\frac{\pi}{2}\sin d)\}^{\alpha+\beta}\cos d}
 + %\rme^{\pi(1-\alpha)/12}\rme^{\pi(\alpha+\beta)/2}
\rme^{\pi(1+5\alpha+6\beta)/12}
%\rme^{\pi/12}
\right\},
\end{align*}
and where $\tilde{c}_{d}$ is a constant defined in~\eqref{tilde-c-alpha-d}.
\end{theorem}

\begin{remark}
The formulas with $\DEt{1}$, $\DEt{2}$, and $\DEt{3}$
were originally developed by Muhammad--Mori~\cite{muhammad03:_doubl},
but no error analysis was done on cases 2 and 3.
In case 1, the authors claimed that
the formula can achieve
$\Order(\exp(-\pi d n/\log(4 d n/(\mu-\epsilon))))$,
where $\epsilon$ denotes an arbitrary small positive number,
under some mild conditions (not specified clearly).
In contrast,
Theorem~\ref{thm:DE1-Sinc-Indef} states the
better convergence rate than that
claimed by Muhammad--Mori~\cite{muhammad03:_doubl},
under explicit assumptions.
\end{remark}

%#!latex 2012scan-error-estim
\section{Numerical Examples} \label{sec:numer-exam}

In order to numerically confirm the results presented
in Sections~\ref{sec:Sinc-quad-estimates}
and~\ref{sec:Sinc-indef-estimates},
let us consider the following three examples.

\begin{example}[Case 1~{\cite{muhammad03:_doubl}}]
\label{Exam:case_1}
Consider the function
$f_1(t)=\sqrt{3}/(2\pi(t^2+t+1))$
and its definite/indefinite integral on $I_1$:
\begin{align}
\int_{-\infty}^{\infty}f_1(t)\diff t&=1,\label{eq:Sinc_quad0}\\
\int_{-\infty}^{\tau}f_1(t)\diff t&=
\frac{1}{2}+\frac{1}{\pi}\arctan\left\{\frac{2}{\sqrt{3}}\left(\tau+\frac{1}{2}\right)\right\}.\label{eq:Sinc_indef0}
%\frac{\arctan(\sqrt{2}t/(1-t^2))}{2\sqrt{2}}
%\frac{\pi+\arctan(\sqrt{2}t-1)+\arctan(\sqrt{2}t+1)}{2\sqrt{2}}
%+\frac{1}{4\sqrt{2}}\log\left(\frac{\tau^2+\sqrt{2}\tau+1}{\tau^2-\sqrt{2}\tau+1}\right).
%+\frac{\pi}{2\sqrt{2}}.
\end{align}
The integrand $f_1$ satisfies the assumptions
in Theorems~\ref{thm:SE1-Sinc-Quad-Explicit}
and~\ref{thm:SE1-Sinc-Indef}
with $\alpha=\beta=1$, $d=3/4$, and $K=\sqrt{3}\rme$.
In addition, $f_1$ satisfies the assumptions
in Theorems~\ref{thm:DE1-Sinc-Quad-explicit}
and~\ref{thm:DE1-Sinc-Indef}
with $\alpha=\beta=1$, $d=\pi/7$, and $K=8\sqrt{3}/\mathrm{e}$.
\end{example}
\begin{example}[Case 2~{\cite{muhammad03:_doubl}}]
\label{Exam:case_2}
Consider the function
$f_2(t)=2/(\pi(1+t^2))$
and its definite/indefinite integral on $I_2$:
\begin{align}
\int_{0}^{\infty}f_2(t)\diff t&=1,\label{eq:Sinc_quad2}\\
\int_{0}^{\tau}f_2(t)\diff t&=\frac{2}{\pi}\arctan(\tau).
\label{eq:Sinc_indef2}
\end{align}
The integrand $f_2$ satisfies the assumptions
in Theorems~\ref{thm:SE2-Sinc-Quad-Explicit}
and~\ref{thm:SE2-Sinc-Indef} ($i=2$)
with $\alpha=\beta=1$, $d=\cosh(1)$, and $K=2/\pi$.
In addition, $f_2$ satisfies the assumptions
in Theorems~\ref{thm:DE2-Sinc-Quad-explicit}
and~\ref{thm:DE2-Sinc-Indef}
with $\alpha=\beta=1$, $d=3/2$, and $K=2/\pi$.
%In addition, $f_2$ satisfies the assumptions
%in Theorem~\ref{thm:DE2-Sinc-Quad-explicit}
%with $\alpha=\beta=1$, $d=3/2$, and $K=2/\pi$.
\end{example}
\begin{example}[Case 3~{\cite{takahasi74:_doubl}}]
\label{Exam:case_3}
Consider the function
$f_3(t)=\rme^{-(1+t)}/(1+t)$
and its definite/indefinite integral on $I_3$:
\begin{align}
\int_{0}^{\infty}f_3(t)\diff t&=\EInt(1),\label{eq:Sinc_quad3}\\
\int_{0}^{\tau}f_3(t)\diff t&=\EInt(1)-\Gamma(0,1+\tau),
\label{eq:Sinc_indef3}
\end{align}
where $\EInt(x)$ is the exponential integral,
%defined by $\EInt(x)=\int_x^{\infty}(\rme^{-t}/t)\diff t$,
and $\Gamma(s,x)$ is the incomplete gamma function.
The integrand $f_3$ satisfies the assumptions
in Theorems~\ref{thm:SE2-Sinc-Quad-Explicit}
and~\ref{thm:SE2-Sinc-Indef} ($i=3$)
with $\alpha=\beta=1$, $d=3/2$, and $K=\rme^{-1}$.
In addition, $f_3$ satisfies the assumptions
in Theorems~\ref{thm:DE3-Sinc-Quad-explicit}
and~\ref{thm:DE3-Sinc-Indef}
with $\alpha=\beta=1$, $d=\log(\pi)$, and $K=\rme$.
\end{example}

Numerical results are shown in
Figures~\ref{Fig:Sinc_quad0}--\ref{Fig:Sinc_indef3}.
All programs were written in C++ with double-precision
floating-point arithmetic,
and the GNU Scientific Library was used for computing special functions
(for this reason, rounding errors are not considered).
In Figure~\ref{Fig:Sinc_indef0},
`maximum error' denotes the maximum value
% value
of absolute errors investigated on the following 403 points:
$
 \tau=0,\,\pm 2^{-100},\,\pm 2^{-99},\,\ldots,\,
 \pm 2^{-1},\,\pm 2^{0}$,
$\pm 2^{1},\,\ldots,\,\pm 2^{100}$.
%$
Similarly, in Figures~\ref{Fig:Sinc_indef2} and~\ref{Fig:Sinc_indef3},
errors were investigated on 201 points
(just the positive points of above), and their maximum is plotted
in those figures.
In each graph,
we can see that the error estimate by the presented theorem (dotted line)
surely bounds the actual error (solid line).

\begin{figure}
%\begin{center}
 \begin{minipage}{0.47\linewidth}
  \includegraphics[width=1.05\linewidth]{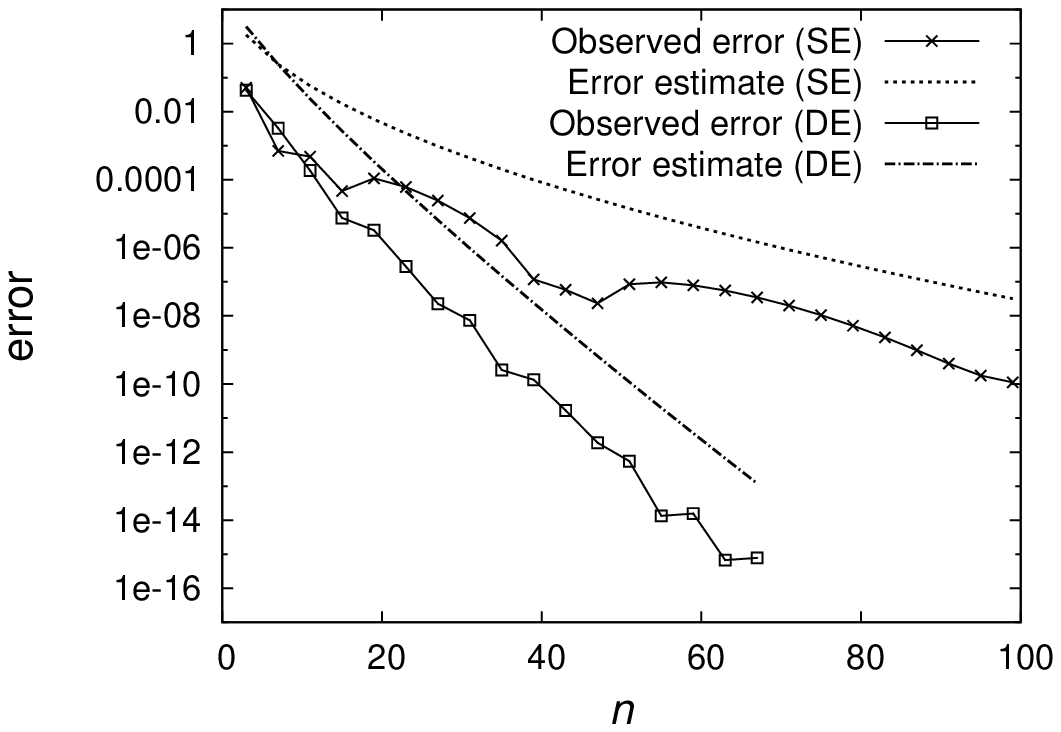}
  \caption{Error of the Sinc quadrature for~\eqref{eq:Sinc_quad0} and its estimate.}
  \label{Fig:Sinc_quad0}
 \end{minipage}
 \begin{minipage}{0.02\linewidth}
  \mbox{ }
 \end{minipage}
 \begin{minipage}{0.47\linewidth}
  \includegraphics[width=1.05\linewidth]{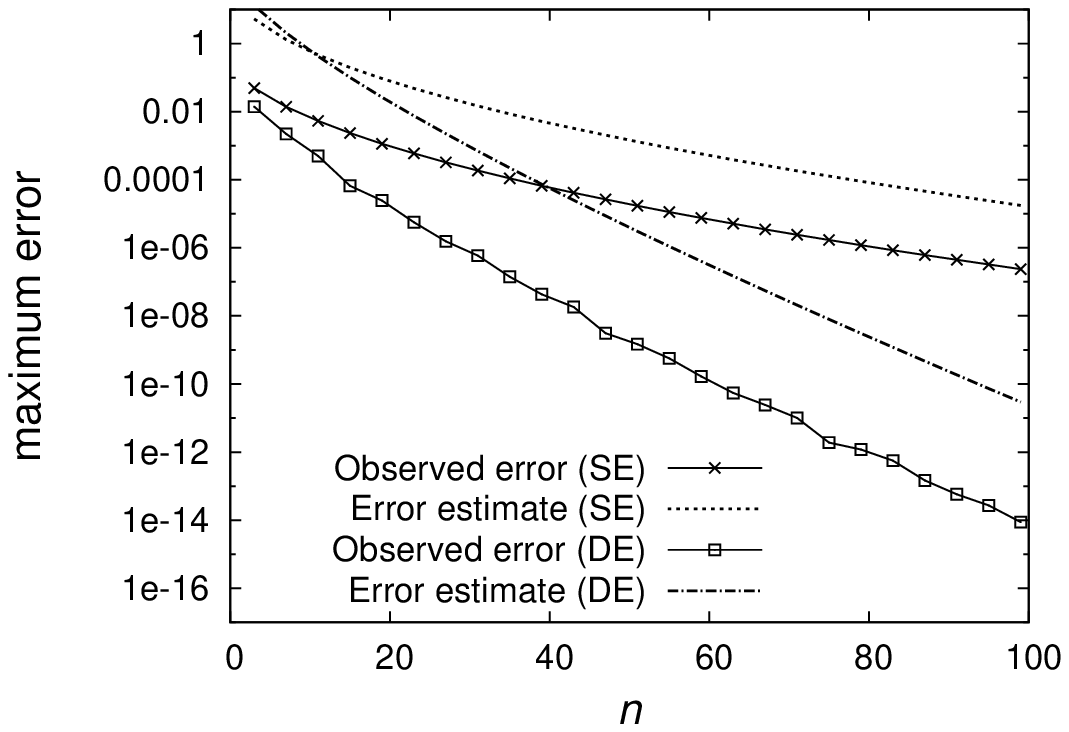}
  \caption{Error of the Sinc indefinite integration for~\eqref{eq:Sinc_indef0} and its estimate.}
  \label{Fig:Sinc_indef0}
 \end{minipage}
%\end{center}

%\begin{center}
 \begin{minipage}{0.47\linewidth}
  \includegraphics[width=1.05\linewidth]{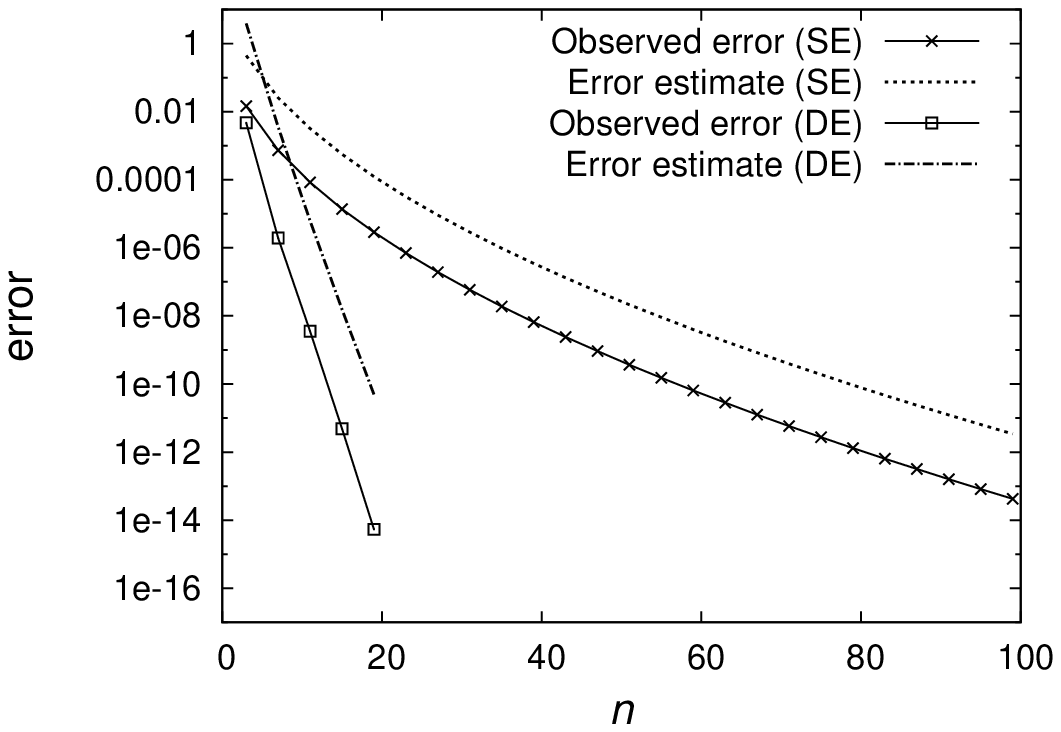}
  \caption{Error of the Sinc quadrature for~\eqref{eq:Sinc_quad2} and its estimate.}
  \label{Fig:Sinc_quad2}
 \end{minipage}
 \begin{minipage}{0.02\linewidth}
  \mbox{ }
 \end{minipage}
 \begin{minipage}{0.47\linewidth}
  \includegraphics[width=1.05\linewidth]{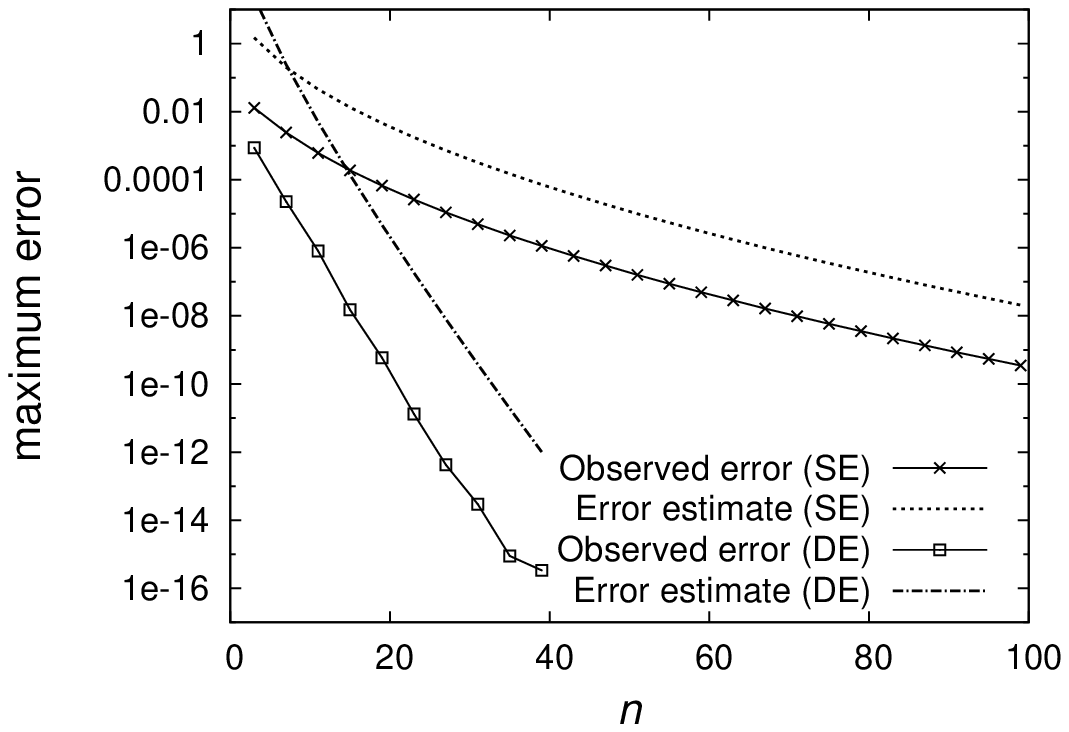}
  \caption{Error of the Sinc indefinite integration for~\eqref{eq:Sinc_indef2} and its estimate.}
  \label{Fig:Sinc_indef2}
 \end{minipage}
%\end{center}

%\begin{center}
 \begin{minipage}{0.47\linewidth}
  \includegraphics[width=1.05\linewidth]{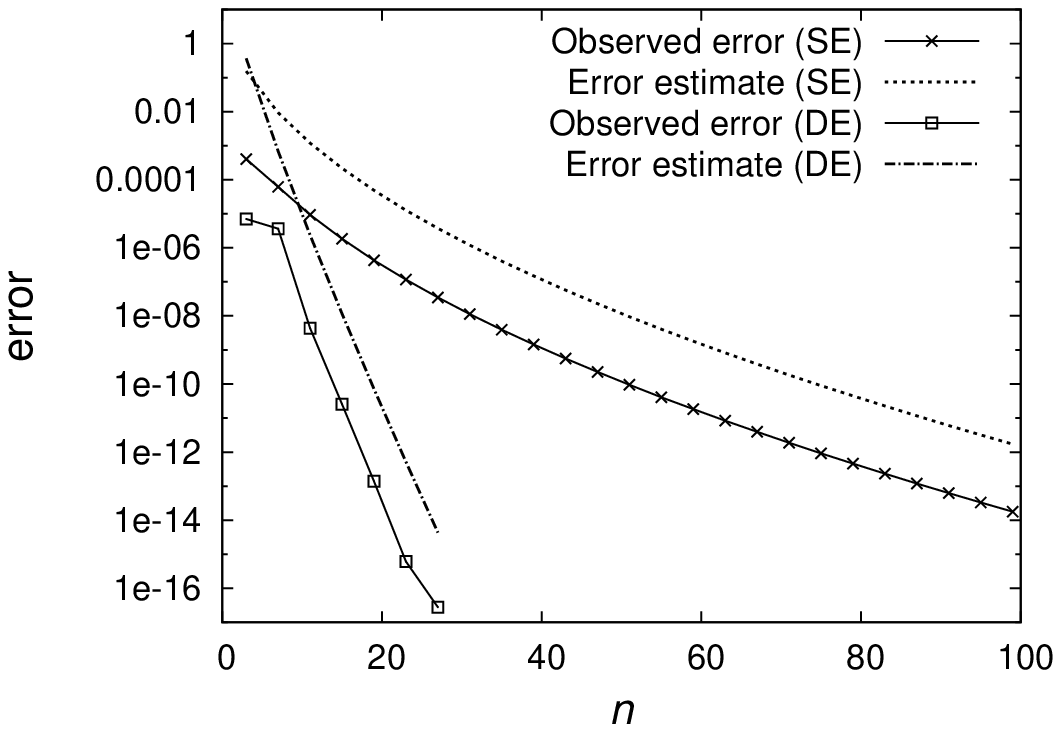}
  \caption{Error of the Sinc quadrature for~\eqref{eq:Sinc_quad3} and its estimate.}
  \label{Fig:Sinc_quad3}
 \end{minipage}
 \begin{minipage}{0.02\linewidth}
  \mbox{ }
 \end{minipage}
 \begin{minipage}{0.47\linewidth}
  \includegraphics[width=1.05\linewidth]{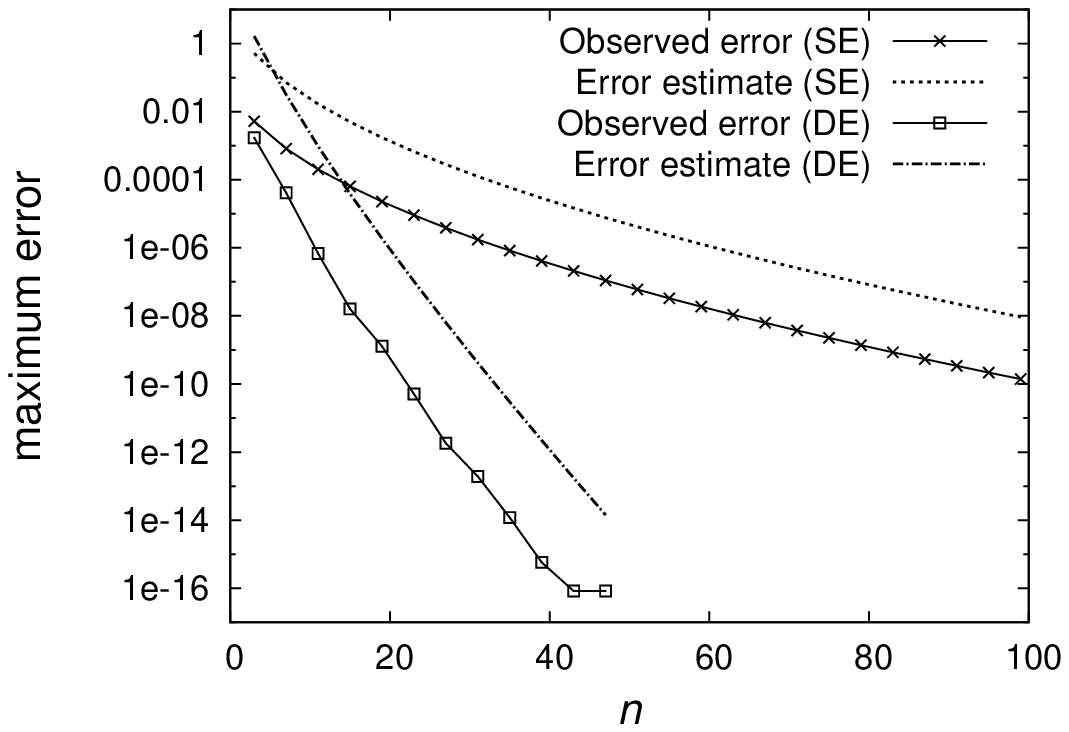}
  \caption{Error of the Sinc indefinite integration for~\eqref{eq:Sinc_indef3} and its estimate.}
  \label{Fig:Sinc_indef3}
 \end{minipage}
%\end{center}
\end{figure}

%#!latex 2012scan-error-estim

\section{Proofs}
\label{sec:proofs}

\subsection{In the Case of the SE Transformation}

Let us first have a look at the sketch of the proof
by using the Sinc quadrature as an example.
%Let $I_1=(-\infty,\,\infty)$, $I_2=I_3=(0,\,\infty)$, and
Let $F(x)=f(\SEt{$i$}(x))\SEtDiv{$i$}(x)$
(recall that we employ the SE transformation $t=\SEt{$i$}(x)$).
Then, we have to evaluate the following term:
%Then the error is divided into two parts as follows:
\begin{align*}
\left|
\int_{I_i}f(t)\diff t
-h\sum_{k=-M}^N f(\SEt{$i$}(kh))\SEtDiv{$i$}(kh)
\right|
=\left|
\int_{-\infty}^{\infty}F(x)\diff x
-h\sum_{k=-M}^N F(kh)
\right|.
%&\leq\left|
%\int_{-\infty}^{\infty}F(x)\diff x
%-h\sum_{k=-\infty}^{\infty} F(kh)
%\right|
%+\left|h\sum_{k<-M}F(kh) + h\sum_{k>N}F(kh)\right|,
\end{align*}
%which are called
%the ``discretization error''
% (say $E_{\mathrm{D}}$),
%and the ``truncation error,'' respectively.
% (say $E_{\mathrm{T}}$)
For the estimation,
the function space defined below plays an important role.
\begin{definition}
Let $L,\,R,\,\alpha,\,\beta$ be positive constants,
and $d$ be a constant where $0<d<\pi/2$.
Then, $\LC_{L,R,\alpha,\beta}^{\textSEg}(\domD_d)$
denotes a family of functions $F$
that are analytic on $\domD_d$,
and for all $\zeta\in\domD_d$ and $x\in\mathbb{R}$, satisfy
\begin{align}
 |F(\zeta)|&\leq
 \frac{L}{|1+\rme^{-2\zeta}|^{\alpha/2}|1+\rme^{2\zeta}|^{\beta/2}},
\label{ineq:LC-SE-complex}\\
%\intertext{and for all $x\in\mathbb{R}$ that}
 |F(x)|&\leq \frac{R}{(1+\rme^{-2x})^{\alpha/2}(1+\rme^{2x})^{\beta/2}}.
\label{ineq:LC-SE-real}
\end{align}
\end{definition}
If $F$ belongs to this function space,
the error of the Sinc quadrature is estimated as below.
The proof is omitted here because it is quite similar to
the existing theorem for case 4~\cite[Theorem~2.6]{okayama09:_error}.
%the discretization error
%and the truncation error are estimated as below.
%\begin{lemma}
%\label{lem:SE-Sinc-Quad-discrete}
%Let $F\in\LC_{L,R,\alpha,\beta}^{\textSEg}(\domD_d)$,
%and let $\mu=\min\{\alpha,\,\beta\}$. Then it holds that
%\begin{equation*}
%\left|
%\int_{-\infty}^{\infty}F(x)\diff x
%-h\sum_{k=-\infty}^{\infty} F(kh)
%\right|
%\leq \frac{4L}{\mu\{\cos d\}^{(\alpha+\beta)/2}}
%\cdot \frac{\rme^{-2\pi d/h}}{1-\rme^{-2\pi d/h}}.
%\end{equation*}
%\end{lemma}
%\begin{lemma}
%\label{lem:SE-Sinc-Quad-truncate}
%Let the assumptions in Lemma~\ref{lem:SE-Sinc-Quad-discrete}
%be fulfilled.
%Furthermore, let $M$ and $N$ be defined as~\eqref{eq:Def-SE-Sinc-MN}.
%Then it holds that
%\begin{equation*}
%\left|
%h\sum_{k=-\infty}^{-M-1} F(kh)
%+h\sum_{k=N+1}^{\infty} F(kh)
%\right|
%\leq \frac{2R}{\mu}\rme^{-\mu n h}.
%\end{equation*}
%\end{lemma}
%The proofs are omitted
%since those are similar to the
%existing lemmas~\cite[Lemmas~4.8~and~4.9]{okayama09:_error}.
%As a consequence of those lemmas,
%we have the next theorem.
\begin{theorem}
\label{thm:SE-Sinc-Quad-overall}
Let $F\in\LC_{L,R,\alpha,\beta}^{\textSEg}(\domD_d)$,
let $\mu=\min\{\alpha,\,\beta\}$,
% Then it holds that
%Let the assumptions in Lemma~\ref{lem:SE-Sinc-Quad-truncate}
%be fulfilled.
let $h$ be defined as~\eqref{eq:Def-SE-h},
and let $M$ and $N$ be defined as~\eqref{eq:Def-SE-Sinc-MN}.
Then, it holds that
\begin{equation*}
\left|
\int_{-\infty}^{\infty}F(x)\diff x
-h\sum_{k=-M}^{N} F(kh)
\right|
\leq \frac{2}{\mu}
\left[
\frac{2L}{(1-\rme^{-\sqrt{2\pi d \mu}})\{\cos d\}^{(\alpha+\beta)/2}}
+R
\right]\rme^{-\sqrt{2\pi d \mu n}}.
\end{equation*}
\end{theorem}
This theorem states the desired error estimates with explicit constants
for the Sinc quadrature
(if $F\in\LC_{L,R,\alpha,\beta}^{\textSEg}(\domD_d)$,
which is not yet confirmed).

For the Sinc indefinite integration,
the next theorem holds.
This proof is also omitted because
it is quite similar to the
existing theorem for case 4~\cite[Theorem~2.9]{okayama09:_error}.
%\begin{lemma}
%\label{lem:SE-Sinc-Indef-discrete}
%Let $F\in\LC_{L,R,\alpha,\beta}^{\textSEg}(\domD_d)$,
%and let $\mu=\min\{\alpha,\,\beta\}$. Then it holds that
%\begin{equation*}
%\sup_{\xi\in\mathbb{R}}\left|
%\int_{-\infty}^{\xi}F(x)\diff x
%-\sum_{k=-\infty}^{\infty} F(kh) J(k,h)(\xi)
%\right|
%\leq \frac{4L}{2 d \mu\{\cos d\}^{(\alpha+\beta)/2}}
%\cdot \frac{h\rme^{-\pi d/h}}{1-\rme^{-2\pi d/h}}.
%\end{equation*}
%\end{lemma}
%\begin{lemma}
%\label{lem:SE-Sinc-Indef-truncate}
%Let the assumptions in Lemma~\ref{lem:SE-Sinc-Indef-discrete}
%be fulfilled.
%Furthermore, let $M$ and $N$ be defined as~\eqref{eq:Def-SE-Sinc-MN}.
%Then it holds that
%\begin{equation*}
%\sup_{\xi\in\mathbb{R}}
%\left|
%\sum_{k=-\infty}^{-M-1} F(kh) J(k,h)(\xi)
%+ \sum_{k=N+1}^{\infty} F(kh) J(k,h)(\xi)
%\right|
%\leq \frac{2.2R}{\mu}\rme^{-\mu n h}.
%\end{equation*}
%\end{lemma}
%From those lemmas, we have the next theorem.
\begin{theorem}
\label{thm:SE-Sinc-Indef-overall}
%Let the assumptions in Lemma~\ref{lem:SE-Sinc-Indef-truncate}
%be fulfilled.
%Furthermore, let $h$ be defined as~\eqref{eq:Def-SE-h}.
Let $F\in\LC_{L,R,\alpha,\beta}^{\textSEg}(\domD_d)$,
let $\mu=\min\{\alpha,\,\beta\}$,
let $h$ be defined as~\eqref{eq:Def-SE-h-indef},
and let $M$ and $N$ be defined as~\eqref{eq:Def-SE-Sinc-MN}.
Then, it holds that
\begin{align*}
&\sup_{\xi\in\mathbb{R}}\left|
\int_{-\infty}^{\xi}F(x)\diff x
-\sum_{k=-M}^{N} F(kh)J(k,h)(\xi)
\right|\\
&\leq \frac{2}{\mu}
\left[
\frac{L}{(1-\rme^{-2\sqrt{\pi d \mu}})\{\cos d\}^{(\alpha+\beta)/2}}
\sqrt{\frac{\pi}{d \mu}}
+1.1R
\right]\rme^{-\sqrt{\pi d \mu n}}.
\end{align*}
\end{theorem}
This theorem states the desired error estimates with explicit constants
for the Sinc indefinite integration.

%Therefore,
%what is left is to check the assumptions
%in Theorem~\ref{thm:SE-Sinc-Quad-overall}
%for each case: 1, 2, 3.

In view of Theorems~\ref{thm:SE-Sinc-Quad-overall}
and~\ref{thm:SE-Sinc-Indef-overall},
our project is completed by checking
$F\in\LC_{L,R,\alpha,\beta}^{\textSEg}(\domD_d)$
in each case: 1, 2, and 3.
Let us check each case one by one.

\subsubsection{Proofs in Case 1
(Theorems~\ref{thm:SE1-Sinc-Quad-Explicit}
and~\ref{thm:SE1-Sinc-Indef})}

The claims of
Theorems~\ref{thm:SE1-Sinc-Quad-Explicit}
and~\ref{thm:SE1-Sinc-Indef} follow
from the next lemma.

\begin{lemma}
Let the assumptions in Theorem~\ref{thm:SE1-Sinc-Quad-Explicit}
or Theorem~\ref{thm:SE1-Sinc-Indef}
be fulfilled.
Then, the function $F(\zeta)=f(\SEt{1}(\zeta))\SEtDiv{1}(\zeta)$
belongs to $\LC_{L,R,\alpha,\beta}^{\textSEg}(\domD_d)$
with $L=2^{\nu}K/\{\cos d\}^{(\nu-\mu)/2}$
and $R=2^{\nu}K$.
\end{lemma}
%\begin{lemma}
%Let the assumptions in Theorem~\ref{thm:SE1-Sinc-Indef}
%be fulfilled.
%Then the function $F(\zeta)=f(\SEt{1}(\zeta))\SEtDiv{1}(\zeta)$
%belongs to $\LC_{L,R,\alpha,\beta}^{\textSEg}(\domD_d)$
%with $L=2^{\nu}K/\{\cos d\}^{(\nu-\mu)/2}$
%and $R=2^{\nu}K$.
%\end{lemma}
\begin{proof}
First, consider the case $\alpha\leq \beta$.
From the inequality~\eqref{leq:Sinc-quad-case1-alpha}, it follows that
%it follows that
\begin{align*}
|F(\zeta)|&\leq
%\frac{K|\SEtDiv{1}(\zeta)|}{|1+\{\SEt{1}(\zeta)\}^2|^{(\alpha+1)/2}}
K|E_1(\SEt{1}(\zeta);\alpha)||\SEtDiv{1}(\zeta)|
=
\frac{K}{|1+\rme^{-2\zeta}|^{\alpha/2}|1+\rme^{2\zeta}|^{\beta/2}}
\cdot 2^{\alpha}|1+\rme^{2\zeta}|^{(\beta-\alpha)/2}
%\quad (\Re\zeta<0),
\intertext{for $\zeta\in\domD_d$ with $\Re\zeta<0$,
and from the inequality~\eqref{leq:Sinc-quad-case1-beta}, it
follows that}
|F(\zeta)|&
%\leq\frac{K|\SEtDiv{1}(\zeta)|}{|1+\{\SEt{1}(\zeta)\}^2|^{(\beta+1)/2}}
\leq K|E_1(\SEt{1}(\zeta);\beta)||\SEtDiv{1}(\zeta)|
=
\frac{K}{|1+\rme^{-2\zeta}|^{\alpha/2}|1+\rme^{2\zeta}|^{\beta/2}}
\cdot
\frac{2^{\beta}}{|1+\rme^{-2\zeta}|^{(\beta-\alpha)/2}}
\end{align*}
for $\zeta\in\domD_d$ with $\Re\zeta\geq 0$.
Setting $\zeta=x+\imnum y$ with $x<0$, we have
\begin{align*}
2^{\alpha}|1+\rme^{2\zeta}|^{(\beta-\alpha)/2}
&=2^{\alpha}(1+\rme^{2x})^{(\beta-\alpha)/2}
\left\{1-\frac{\sin^2 y}{\cosh^2 x}\right\}^{(\beta-\alpha)/4}\\
&\leq 2^{\alpha}(1+\rme^{0})^{(\beta-\alpha)/2}
\left\{1- 0\right\}^{(\beta-\alpha)/4}\\
&=2^{(\alpha+\beta)/2}\leq 2^{\beta}
\leq \frac{2^{\beta}}{\{\cos y\}^{(\beta-\alpha)/2}}.
\end{align*}
Furthermore, setting $\zeta=x+\imnum y$ with $x\geq 0$, we have
\begin{align*}
\frac{2^{\beta}}{|1+\rme^{-2\zeta}|^{(\beta-\alpha)/2}}
&=\frac{2^{\beta}}{(1+\rme^{-2x})^{(\beta-\alpha)/2}
\left\{1-\frac{\sin^2 y}{\cosh^2 x}\right\}^{(\beta-\alpha)/4}}\\
&\leq
\frac{2^{\beta}}{(1+ 0)^{(\beta-\alpha)/2}
\left\{1-\frac{\sin^2 y}{\cosh^2 0}\right\}^{(\beta-\alpha)/4}}
=\frac{2^{\beta}}{\{\cos y\}^{(\beta-\alpha)/2}}.
\end{align*}
Thus, because $\mu=\alpha$ and $\nu=\beta$ in this case,
it holds for all $\zeta\in\domD_d$ that
\begin{align*}
|F(\zeta)|&\leq
\frac{K}{|1+\rme^{-2\zeta}|^{\alpha/2}|1+\rme^{2\zeta}|^{\beta/2}}
\cdot
\frac{2^{\nu}}{\{\cos d\}^{(\nu-\mu)/2}},
\intertext{and it holds for all $x\in\mathbb{R}$ that}
%\begin{align*}
|F(x)|&\leq
\frac{K}{(1+\rme^{-2x})^{\alpha/2}(1+\rme^{2 x})^{\beta/2}}
\cdot
\frac{2^{\nu}}{\{\cos 0\}^{(\nu-\mu)/2}}.
\end{align*}
%holds for all $x\in\mathbb{R}$.
In the case $\alpha>\beta$, the same inequalities hold.
This completes the proof.
%Hence the claim follows.
\end{proof}

\subsubsection{Proofs in Case 2
(Theorems~\ref{thm:SE2-Sinc-Quad-Explicit}
and~\ref{thm:SE2-Sinc-Indef} with $i=2$)}

The claims of
Theorems~\ref{thm:SE2-Sinc-Quad-Explicit}
and~\ref{thm:SE2-Sinc-Indef} ($i=2$) follow
from the next lemma.

\begin{lemma}
Let the assumptions in Theorem~\ref{thm:SE2-Sinc-Quad-Explicit}
or Theorem~\ref{thm:SE2-Sinc-Indef}
be fulfilled with $i=2$.
Then, the function $F(\zeta)=f(\SEt{2}(\zeta))\SEtDiv{2}(\zeta)$
belongs to $\LC_{L,R,\alpha,\beta}^{\textSEg}(\domD_d)$
with $L=K$ and $R=K$.
\end{lemma}
%\begin{lemma}
%Let the assumptions in Theorem~\ref{thm:SE2-Sinc-Indef}
%be fulfilled with $i=2$.
%Then the function $F(\zeta)=f(\SEt{2}(\zeta))\SEtDiv{2}(\zeta)$
%belongs to $\LC_{L,R,\alpha,\beta}^{\textSEg}(\domD_d)$
%with $L=K$ and $R=K$.
%\end{lemma}
\begin{proof}
From the inequality~\eqref{leq:Sinc-quad-case2-alpha-beta} with $i=2$,
\eqref{ineq:LC-SE-complex}
and~\eqref{ineq:LC-SE-real} immediately hold with $L=R=K$.
%it follows that
%\begin{align*}
%|F(\zeta)|\leq
%% K\frac{|\rme^{\zeta}|^{\alpha-1}}{|1+\rme^{2 \zeta}|^{(\alpha+\beta)/2}}|\SEtDiv{2}(\zeta)|
%K|E_2(\SEt{2}(\zeta);\alpha,\beta)||\SEtDiv{2}(\zeta)|
%=\frac{K}{|1+\rme^{-2\zeta}|^{\alpha/2}|1+\rme^{2\zeta}|^{\beta/2}}
%\end{align*}
%for all $\zeta\in\domD_d$
%(which clearly includes $\zeta\in\mathbb{R}$). This completes the proof.
\end{proof}

\subsubsection{Proofs in Case 3
(Theorems~\ref{thm:SE2-Sinc-Quad-Explicit}
and~\ref{thm:SE2-Sinc-Indef} with $i=3$)}

The claims of
Theorems~\ref{thm:SE2-Sinc-Quad-Explicit}
and~\ref{thm:SE2-Sinc-Indef} ($i=3$) follow
from the next lemma.

\begin{lemma}
\label{lem:SE3-Sinc-Quad-check}
Let the assumptions in Theorem~\ref{thm:SE2-Sinc-Quad-Explicit}
or Theorem~\ref{thm:SE2-Sinc-Indef}
be fulfilled with $i=3$.
Then, the function $F(\zeta)=f(\SEt{3}(\zeta))\SEtDiv{3}(\zeta)$
belongs to $\LC_{L,R,\alpha,\beta}^{\textSEg}(\domD_d)$
with $L=2^{\beta/2}c_{\alpha,d} K$
and $R=2^{(1-\alpha+|1-\alpha|)/2}K$,
where $c_{\alpha,d}$ is the constant defined in~\eqref{eq:def-c_alpha_d}.
\end{lemma}
%\begin{lemma}
%\label{lem:SE3-Sinc-Indef-check}
%Let the assumptions in Theorem~\ref{thm:SE2-Sinc-Indef}
%be fulfilled with $i=3$.
%Then the function $F(\zeta)=f(\SEt{2}(\zeta))\SEtDiv{2}(\zeta)$
%belongs to $\LC_{L,R,\alpha,\beta}^{\textSEg}(\domD_d)$
%with $L=2^{(\alpha+\beta-1)/2}c_{\alpha,d}K$
%and $R=K$,
%where $c_{\alpha,d}$ is a constant defined in~\eqref{eq:def-c_alpha_d}.
%\end{lemma}

For the proof,
let us prepare some useful inequalities
(Lemmas~\ref{lem:asinh-SE}--\ref{lem:cosh-2-SE}).

\begin{lemma}
\label{lem:asinh-SE}
For all $\zeta\in\overline{\domD_{\pi/2}}$,
it holds that
\begin{equation}
\frac{1}{\sqrt{2}}
\left|\frac{\rme^{\zeta}}{1+\rme^{\zeta}}\right|
\leq\left|
\frac{\arcsinh(\rme^{\zeta})}{1+\arcsinh(\rme^{\zeta})}
\right|
\leq\sqrt{2}
\left|\frac{\rme^{\zeta}}{1+\rme^{\zeta}}\right|.
\label{ineq:asinh-complex}
\end{equation}
Furthermore, for all $x\in\mathbb{R}$,
it holds that
\begin{equation}
\frac{\arcsinh(\rme^{x})}{1+\arcsinh(\rme^{x})}
\leq\frac{\rme^{x}}{1+\rme^{x}}.
\label{ineq:asinh-real}
\end{equation}
\end{lemma}
\begin{proof}
First, consider~\eqref{ineq:asinh-real},
which is proved by showing that
$p(t)\leq p(\sinh t)$ for $t\geq 0$, where
$p(t)=t/(1+t)$ (just put $x=\log(\sinh t)$).
%\[
%\frac{t}{1+t}\leq \frac{\sinh t}{1+\sinh t}
%\]
%for $t\geq 0$ (put $x=\log(\sinh t)$).
%Set $p(t)=t/(1+t)$.
Because $p(t)$ is monotonically increasing,
the desired inequality
$p(t)\leq p(\sinh t)$ clearly holds.
Hence,~\eqref{ineq:asinh-real} is proved.

Next, consider~\eqref{ineq:asinh-complex},
which is proved
by showing that $|g(\zeta)|\leq\sqrt{2}$ and
$|1/g(\zeta)|\leq\sqrt{2}$, where
\[
 g(\zeta)=\frac{\arcsinh(\rme^{\zeta})}{1+\arcsinh(\rme^{\zeta})}
\frac{1+\rme^{\zeta}}{\rme^{\zeta}}.
\]
The functions $g$ and $1/g$ are analytic on $\domD_{\pi/2}$
(and continuous on $\overline{\domD_{\pi/2}}$).
Therefore, by the maximum modulus principle,
$|g(\zeta)|$ and $|1/g(\zeta)|$ have their maximum
on the boundary of $\domD_{\pi/2}$.
%i.e., $|\Im\zeta|=\pi/2$.
It is sufficient to consider $z=x+\imnum (\pi/2)$ for $x\in\mathbb{R}$.

First, let us show $1/|g(z)|\leq \sqrt{2}$,
which is relatively easy.
Setting $X=\Re\{\arcsinh(\mathrm{i}\rme^x)\}$ and
$Y=\Im\{\arcsinh(\mathrm{i}\rme^x)\}$, we have
%\log[|1-\rme^{2x}|+\rme^{2 x}]$ and
%$Y=\arg[\sqrt{1-\rme^{2x}}+\imnum \rme^{x}]$, we have
\[
|g(z)|^2=|g(x+\imnum \pi/2)|^2
=(1+\rme^{-2x})\frac{X^2+Y^2}{(X+1)^2+Y^2}.
\]
Furthermore, from
\begin{align*}
X&=
\begin{cases}
\log[\rme^{x}+\sqrt{\rme^{2x}-1}]&(\text{if}\,\,\,x\geq 0),\\
0&(\text{if}\,\,\,x< 0),
\end{cases}\\
Y&=
\begin{cases}
\pi/2 & (\text{if}\,\,\,x\geq 0),\\
\arctan(\rme^{x}/\sqrt{1-\rme^{2x}})&(\text{if}\,\,\,x<0),
\end{cases}
\end{align*}
it holds for $x\geq 0$ that
\begin{align*}
\frac{1}{|g(z)|^2}
&=\frac{1}{1+\rme^{-2x}}\cdot
\frac{(1+\log[\rme^x+\sqrt{\rme^{2x}-1}])^2+(\pi/2)^2}{\log^2[\rme^x+\sqrt{\rme^{2x}-1}]+(\pi/2)^2}\\
&\leq \frac{1}{1+\rme^{-2x}}\cdot
\frac{(1+\{\sqrt{1+\pi^2}-1\}/2)^2+(\pi/2)^2}{(\{\sqrt{1+\pi^2}-1\}/2)^2+(\pi/2)^2}<2,
\end{align*}
and it holds for $x< 0$ that
\begin{align*}
\frac{1}{|g(z)|^2}
=\frac{1}{1+\rme^{-2x}}
\left\{1+\frac{1}{\arctan^2[\rme^x/\sqrt{1-\rme^{2x}}]}\right\}
\leq\frac{1}{1+\rme^{-2x}}
\left\{1+\frac{1}{(\rme^x)^2}\right\}=1.
\end{align*}
Thus, $|1/g(x+\imnum \pi/2)|^2\leq 2$
holds for all $x\in\mathbb{R}$.

Next, let us show $|g(z)|\leq \sqrt{2}$.
It holds for $x\geq 0$ that
\[
 |g(z)|^2
=(1+\rme^{-2x})
\frac{\log^2[\rme^x+\sqrt{\rme^{2x}-1}]+(\pi/2)^2}{(1+\log[\rme^x+\sqrt{\rme^{2x}-1}])^2+(\pi/2)^2}
\leq (1+1)\cdot 1 = 2.
\]
For $x<0$, we have
\[
 |g(z)|^2
=(1+\rme^{-2x})
\frac{\arctan^2[\rme^x/\sqrt{1-\rme^{2x}}]}{1+\arctan^2[\rme^x/\sqrt{1-\rme^{2x}}]}
=\frac{2s^2}{1+s^2} + \left(\frac{s}{\tan s}\right)^2\frac{1}{1+s^2},
\]
where $0\leq s=\arctan[\rme^x/\sqrt{1-\rme^{2x}}]<\pi/2$.
In the case where $0\leq s< 1$, it holds that
\[
 |g(z)|^2\leq \frac{2\cdot 1^2}{1+1^2}
 + \left(\frac{s}{\tan s}\right)^2\frac{1}{1+s^2}
\leq \frac{2\cdot 1^2}{1+1^2} + \left(1\right)^2\frac{1}{1+0^2} = 2,
\]
and in the case where $1\leq s<\pi/2$, it holds that
\[
 |g(z)|^2\leq \frac{2\cdot (\pi/2)^2}{1+(\pi/2)^2}
 + \left(\frac{s}{\tan s}\right)^2\frac{1}{1+s^2}
\leq \frac{2\cdot (\pi/2)^2}{1+(\pi/2)^2}+
\left(\frac{1}{\tan 1}\right)^2\frac{1}{1+1^2} < 2.
\]
%Let us bound the right hand side
%in the following three cases:
%\begin{enumerate}
% \item[a)] $-\frac{1}{2}\log[1+\frac{8}{\pi^2}]\leq x<0$,
% \item[b)] $-\frac{1}{2}\log[1+\frac{18}{\pi^2}]\leq x\leq -\frac{1}{2}\log[\frac{4}{3}]$
%\quad (note that $-\frac{1}{2}\log[1+\frac{8}{\pi^2}]<-\frac{1}{2}\log[\frac{4}{3}]$),
% \item[c)] $x<-\frac{1}{2}\log[1+\frac{18}{\pi^2}]$.
%\end{enumerate}
%In the case of a), it holds that
%\[
% |g(z)|^2
%\leq (1+\rme^{-2x})\frac{(\pi/2)^2}{1+(\pi/2)^2}
%\leq\left(1+\left(1+\frac{8}{\pi^2}\right)\right)
%\frac{(\pi/2)^2}{1+(\pi/2)^2} = 2.
%\]
%In the case of b), it holds that
%\[
% |g(z)|^2
%\leq (1+\rme^{-2x})\frac{(\pi/3)^2}{1+(\pi/3)^2}
%\leq\left(1+\left(1+\frac{18}{\pi^2}\right)\right)
%\frac{(\pi/3)^2}{1+(\pi/3)^2} = 2.
%\]
%In the case of c), it holds that
%\begin{align*}
%|g(z)|^2
%&\leq (1+\rme^{-2x})\frac{[\rme^x/\sqrt{1-\rme^{2x}}]^2}{1+[\rme^{x}/\sqrt{1-0}]^2}\\
%&=\frac{1}{2}\left(1-\frac{1}{\tanh x}\right)
%\leq\frac{1}{2}\left(1-\frac{1}{\tanh (-\frac{1}{2}\log[1+\frac{18}{\pi^2}])}\right)
%=1+\frac{\pi^2}{18}<2.
%\end{align*}
Thus, $|g(x+\imnum \pi/2)|^2\leq 2$ for all $x\in\mathbb{R}$.
This completes the proof.
\end{proof}

\begin{lemma}
\label{lem:exp-asinh-SE}
For all $\zeta\in\overline{\domD_{\pi/2}}$
and $x\in\mathbb{R}$, we have
\begin{align}
\frac{1}{|\rme^{\zeta}+\sqrt{1+\rme^{2\zeta}}|}
&\leq\frac{\sqrt{2}}{|1+\rme^{\zeta}|},
\label{ineq:exp-asinh-complex}\\
\frac{1}{\rme^{x}+\sqrt{1+\rme^{2x}}}
&\leq\frac{1}{1+\rme^{x}}.
\label{ineq:exp-asinh-real}
\end{align}
\end{lemma}
\begin{proof}
First, consider~\eqref{ineq:exp-asinh-real},
which is proved by showing
\[
 g(x)=\frac{1+\rme^x}{\rme^x + \sqrt{1+\rme^{2 x}}}\leq 1
\]
for $x\in\mathbb{R}$.
Because
\[
 g'(x)=-\rme^x\left(1-\frac{\rme^x}{\sqrt{1+\rme^{2x}}}\right)
\left((1+\rme^x)-\sqrt{1+\rme^{2x}}\right) \leq 0,
\]
we have $g(x)\leq \lim_{x\to-\infty}\{g(x)\}=1$, which is the desired result.

Next, consider~\eqref{ineq:exp-asinh-complex},
which is proved
by showing $|g(\zeta)|\leq \sqrt{2}$.
The function $g$ is analytic on $\domD_{\pi/2}$,
(and continuous on $\overline{\domD_{\pi/2}}$).
Therefore, by the maximum modulus principle,
$|g(\zeta)|$ has its maximum
on the boundary of $\domD_{\pi/2}$,
i.e., $|\Im\zeta|=\pi/2$.
It is sufficient to consider $z=x+\imnum (\pi/2)$ for $x\in\mathbb{R}$.
In the case $x\geq 0$, we have
\begin{align*}
|g(z)|^2
=\left|
\frac{1}{1+\sqrt{1-\rme^{-2x}}}-\imnum
\frac{\rme^{-x}}{1+\sqrt{1-\rme^{-2x}}}
\right|^2
=\frac{1+\rme^{-2x}}{(1+\sqrt{1-\rme^{-2x}})^2}
\leq 1+\rme^{-2x}\leq 2.
\end{align*}
In the case $x< 0$, we have
\begin{align*}
|g(z)|^2
=\left|
\sqrt{1-\rme^{2x}}+\rme^{2x}
-\imnum(1-\sqrt{1-\rme^{2x}})\rme^{x}
\right|^2
=1+\rme^{2x}\leq 2.
\end{align*}
This completes the proof.
\end{proof}

\begin{lemma}
\label{lem:exp-2-SE}
For all $\zeta\in\domD_{\pi/2}$,
it holds that
\begin{equation*}
\frac{1}{|1+\rme^{\zeta}|}
\leq\frac{1}{|1+\rme^{2\zeta}|^{1/2}}.
%\label{ineq:exp-asinh-complex}
\end{equation*}
\end{lemma}
\begin{proof}
Let $x,\,y\in\mathbb{R}$ with $|y|<\pi/2$,
and let $\zeta=x+\imnum y$. Then, we have
\[
 \left|\frac{1}{1+\rme^{2\zeta}}\right|
-\left|\frac{1}{1+\rme^{\zeta}}\right|^2
=\frac{\rme^{-x}}{2(\cosh(x)+\cos(y))}
%\sqrt{\frac{1+2\cosh(x)\cos(y)}{\cosh^2(x)-\sin^2(y)}}\geq 0,
\left(
\sqrt{\frac{2(\cosh(x)+\cos(y))^2}{\cosh(2x)+\cos(2y)}}-1
\right).
\]
%which completes the proof.
To show that the right-hand side is nonnegative,
it suffices to observe
\[
 \frac{2(\cosh(x)+\cos(y))^2}{\cosh(2x)+\cos(2y)}-1
=\frac{1+2\cosh(x)\cos(y)}{\cosh^2(x)-\sin^2(y)}
\geq 0,
\]
because it generally holds that
if $t-1\geq 0$ then $\sqrt{t}-1=(t-1)/(\sqrt{t}+1)\geq 0$.
%This completes the proof.
\end{proof}

\begin{lemma}
\label{lem:cosh-2-SE}
Let $x,\,y\in\mathbb{R}$ with $|y|<\pi/2$.
Then, it holds that
\[
 \sup_{x\in\mathbb{R}}
\left|\frac{\cosh^2[(x+\imnum y)/2]}{\cosh(x+\imnum y)}\right|
\leq
\frac{\cos^2(y/2)}{\cos y}
=
\frac{1}{2}\left(1+\frac{1}{\cos y}\right).
\]
\end{lemma}
\begin{proof}
First, we have
\[
 \left|\frac{\cosh^2[(x+\imnum y)/2]}{\cosh(x+\imnum y)}\right|
=\frac{\cosh(x)+\cos(y)}{2\sqrt{\cosh^2(x)-\sin^2(y)}}
\leq\frac{\cosh(x)+\cos(y)}{2\sqrt{\cosh^2(x)\{1-\sin^2(y)\}}},
\]
which is then estimated as
\[
\frac{\cosh(x)+\cos(y)}{2\sqrt{\cosh^2(x)\{1-\sin^2(y)\}}}
=
\frac{\cosh(x)+\cos(y)}{2\cosh(x)\cos(y)}
\leq \frac{1+\cos(y)}{2\cdot 1\cdot\cos(y)}.
\]
This completes the proof.
\end{proof}

Lemma~\ref{lem:SE3-Sinc-Quad-check}
%and~\ref{lem:SE3-Sinc-Indef-check}
is proved as follows.

\begin{proof}
From the inequality~\eqref{leq:Sinc-quad-case2-alpha-beta} with $i=3$,
it follows that
\begin{align*}
|F(\zeta)|\leq
%K |E_3(\SEt{3}(\zeta);\alpha,\beta)||\SEtDiv{3}(\zeta)|=
 K
\left|\frac{\arcsinh(\rme^{\zeta})}{1+\arcsinh(\rme^{\zeta})}\right|^{\alpha-1}
\left|\frac{1}{\rme^{\zeta}+\sqrt{1+\rme^{2\zeta}}}\right|^{\beta}
\left|\frac{1}{1+\rme^{-2\zeta}}\right|^{1/2}.
\end{align*}
First, consider the case $1\leq\alpha$.
From Lemmas~\ref{lem:asinh-SE},~\ref{lem:exp-asinh-SE},
and~\ref{lem:exp-2-SE}, it holds that
\begin{align*}
|F(\zeta)|
&\leq K\left|\frac{\sqrt{2}}{1+\rme^{-\zeta}}\right|^{\alpha-1}
\left|\frac{\sqrt{2}}{1+\rme^{\zeta}}\right|^{\beta}
\left|\frac{1}{1+\rme^{-2\zeta}}\right|^{1/2}\\
&\leq 2^{\beta/2}2^{(\alpha-1)/2}K
\left|\frac{1}{1+\rme^{-2\zeta}}\right|^{(\alpha-1)/2}
\left|\frac{1}{1+\rme^{2\zeta}}\right|^{\beta/2}
\left|\frac{1}{1+\rme^{-2\zeta}}\right|^{1/2}
%&=2^{(\alpha+\beta-1)/2}K
%\left|\frac{1}{1+\rme^{-\zeta}}\right|^{\alpha/2}
%\left|\frac{1}{1+\rme^{\zeta}}\right|^{\beta/2}
\end{align*}
for all $\zeta\in\domD_d$. For $x\in\mathbb{R}$, it holds that
\begin{align*}
|F(x)|
&\leq K\left(\frac{1}{1+\rme^{-x}}\right)^{\alpha-1}
\left(\frac{1}{1+\rme^{x}}\right)^{\beta}
\left(\frac{1}{1+\rme^{-2x}}\right)^{1/2}\\
&\leq K \left(\frac{1}{1+\rme^{-2x}}\right)^{(\alpha-1)/2}
\left(\frac{1}{1+\rme^{2x}}\right)^{\beta/2}
\left(\frac{1}{1+\rme^{-2x}}\right)^{1/2}.
\end{align*}
This completes the proof for $1\leq\alpha$.

Next, consider the case $0<\alpha< 1$.
From Lemmas~\ref{lem:asinh-SE},~\ref{lem:exp-asinh-SE},
and~\ref{lem:exp-2-SE}, it holds that
\begin{align*}
|F(\zeta)|
&\leq K\left|\sqrt{2}(1+\rme^{-\zeta})\right|^{1-\alpha}
\left|\frac{\sqrt{2}}{1+\rme^{\zeta}}\right|^{\beta}
\left|\frac{1}{1+\rme^{-2\zeta}}\right|^{1/2}\\
&= K\left|\frac{2(1+\rme^{-\zeta})^2}{1+\rme^{-2\zeta}}\right|^{(1-\alpha)/2}
\left|\frac{\sqrt{2}}{1+\rme^{\zeta}}\right|^{\beta}
\left|\frac{1}{1+\rme^{-2\zeta}}\right|^{\alpha/2}.
\end{align*}
Furthermore, from Lemma~\ref{lem:cosh-2-SE}, it holds
for $\zeta=x+\imnum y\in\domD_d$ that
\[
 \left|\frac{2(1+\rme^{-\zeta})^2}{1+\rme^{-2\zeta}}\right|
=4\left|\frac{\cosh^2(\zeta/2)}{\cosh(\zeta)}\right|
\leq 2\left(1+\frac{1}{\cos y}\right).
\]
Then, using Lemma~\ref{lem:exp-2-SE}, we have
\[
 |F(\zeta)|\leq 2^{\beta/2}K
\left\{2\left(1+\frac{1}{\cos d}\right)\right\}^{(1-\alpha)/2}
\left|\frac{1}{1+\rme^{2\zeta}}\right|^{\beta/2}
\left|\frac{1}{1+\rme^{-2\zeta}}\right|^{\alpha/2}
\]
for all $\zeta\in\domD_d$. For $x\in\mathbb{R}$, it holds that
\begin{align*}
|F(x)|
&\leq K\left(\sqrt{2}(1+\rme^{-x})\right)^{1-\alpha}
\left(\frac{1}{1+\rme^{x}}\right)^{\beta}
\left(\frac{1}{1+\rme^{-2x}}\right)^{1/2}\\
&= K\left(\frac{2(1+\rme^{-x})^2}{1+\rme^{-2x}}\right)^{(1-\alpha)/2}
\left(\frac{1}{1+\rme^{x}}\right)^{\beta}
\left(\frac{1}{1+\rme^{-2x}}\right)^{\alpha/2}\\
&\leq K\left(2\left(1+\frac{1}{\cos 0}\right)\right)^{(1-\alpha)/2}
\left(\frac{1}{1+\rme^{2x}}\right)^{\beta/2}
\left(\frac{1}{1+\rme^{-2x}}\right)^{\alpha/2}.
\end{align*}
This completes the proof for $0<\alpha< 1$.
\end{proof}

\subsection{In the Case of the DE Transformation}

We need the following definition
in the case of the DE transformation.
\begin{definition}
Let $L,\,R,\,\alpha,\,\beta$ be positive constants,
and $d$ be a constant with $0<d<\pi/2$.
Then, $\LC_{L,R,\alpha,\beta}^{\textDEg}(\domD_d)$
denotes a family of functions $F$
that are analytic on $\domD_d$,
and for all $\zeta\in\domD_d$ and $x\in\mathbb{R}$, satisfy
\begin{align}
 |F(\zeta)|&\leq \frac{(\pi/2)L|\cosh\zeta|}{|1+\rme^{-\pi\sinh\zeta}|^{\alpha/2}|1+\rme^{\pi\sinh\zeta}|^{\beta/2}},
\label{ineq:LC-DE-complex}\\
%\intertext{and for all $x\in\mathbb{R}$ that}
 |F(x)|&\leq \frac{(\pi/2)R\cosh x}{(1+\rme^{-\pi\sinh x})^{\alpha/2}(1+\rme^{\pi\sinh x})^{\beta/2}}.
\label{ineq:LC-DE-real}
\end{align}
\end{definition}
If $F$ belongs to this function space,
the errors of the Sinc quadrature
and Sinc indefinite integration are estimated as below.
%
%This theorem states the desired error estimates with explicit constants
%for the Sinc quadrature.
%
%For the Sinc indefinite integration,
%the next theorem holds.
%The proofs is also omitted since
%it is quite similar to the
%existing theorem for the case 4~\cite[Theorem~2.16]{okayama09:_error}.
The proofs are omitted because they are quite similar to
the existing theorems for
case~4~\cite[Theorems~2.14~and~2.16]{okayama09:_error}.
\begin{theorem}
\label{thm:DE-Sinc-Quad-overall}
Let $F\in\LC_{L,R,\alpha,\beta}^{\textDEg}(\domD_d)$,
let $\mu=\min\{\alpha,\,\beta\}$,
let $\nu=\max\{\alpha,\,\beta\}$,
% Then it holds that
%Let the assumptions in Lemma~\ref{lem:DE-Sinc-Quad-truncate}
%be fulfilled.
let $h$ be defined as~\eqref{eq:Def-DE-h},
and let $M$ and $N$ be defined as~\eqref{eq:Def-DE-Sinc-MN}.
Furthermore,
let $n$ be taken sufficiently large so that
$n\geq (\nu \rme)/(8 d)$,
$M h \geq x_{\alpha/2}$, and $N h\geq x_{\beta/2}$ hold.
Then,
% it holds that
\begin{align*}
&\left|
\int_{-\infty}^{\infty}F(x)\diff x
-h\sum_{k=-M}^{N} F(kh)
\right|\\
&\leq \frac{2}{\mu}
\left[
\frac{2L}{(1-\rme^{-\pi\mu\rme/4})\{\cos(\frac{\pi}{2}\sin d)\}^{(\alpha+\beta)/2}\cos d}
+R\rme^{\pi\nu/4}
\right]\rme^{-2\pi d n/\log(8 d n/\mu)}.
\end{align*}
\end{theorem}
\begin{theorem}
\label{thm:DE-Sinc-Indef-overall}
%Let the assumptions in Lemma~\ref{lem:DE-Sinc-Indef-truncate}
%be fulfilled.
%Furthermore, let $h$ be defined as~\eqref{eq:Def-DE-h}.
Let $F\in\LC_{L,R,\alpha,\beta}^{\textDEg}(\domD_d)$,
let $\mu=\min\{\alpha,\,\beta\}$,
let $\nu=\max\{\alpha,\,\beta\}$,
let $h$ be defined as~\eqref{eq:Def-DE-h-half},
and let $M$ and $N$ be defined as~\eqref{eq:Def-DE-Sinc-MN}.
Furthermore,
let $n$ be taken sufficiently large so that
$n\geq (\nu \rme)/(4 d)$,
$M h \geq x_{\alpha/2}$, and $N h\geq x_{\beta/2}$ hold.
Then,
% it holds that
\begin{align*}
&\sup_{\xi\in\mathbb{R}}\left|
\int_{-\infty}^{\xi}F(x)\diff x
-\sum_{k=-M}^{N} F(kh)J(k,h)(\xi)
\right|\\
&\leq \frac{2}{\mu d}
\left[
\frac{L}{(1-\rme^{-\pi\mu\rme/2})\{\cos(\frac{\pi}{2}\sin d)\}^{(\alpha+\beta)/2}\cos d}
+R\rme^{\pi(\alpha+\beta)/4}
\right]\epsilon^{\textDEg}_{d,\mu}(n).
\end{align*}
\end{theorem}
%These theorems state the desired error estimates with explicit constants
%for the Sinc quadrature and the Sinc indefinite integration.
%
%Therefore,
%what is left is to check the assumptions
%in Theorem~\ref{thm:DE-Sinc-Quad-overall}
%for each case: 1, 2, 3.

In view of Theorems~\ref{thm:DE-Sinc-Quad-overall}
and~\ref{thm:DE-Sinc-Indef-overall},
our project is completed by checking
$F\in\LC_{L,R,\alpha,\beta}^{\textDEg}(\domD_d)$
in each case: 1, 2, and 3.
Let us check each case one by one.
The next lemma is useful for the proofs.
\begin{lemma}[Okayama et al.~{\cite[Lemma~4.22]{okayama09:_error}}]
\label{lem:DEfunc-estim}
Let $x,\,y\in\mathbb{R}$ with $|y|<\pi/2$, and let $\zeta=x+\imnum y$.
Then,
\begin{align*}
\left|\frac{1}{1+\rme^{\pi\sinh\zeta}}\right|
&\leq\frac{1}{(1+\rme^{\pi\sinh(x)\cos y})\cos(\frac{\pi}{2}\sin y)},\\
\left|\frac{1}{1+\rme^{-\pi\sinh\zeta}}\right|
&\leq\frac{1}{(1+\rme^{-\pi\sinh(x)\cos y})\cos(\frac{\pi}{2}\sin y)}.
\end{align*}
\end{lemma}

\subsubsection{Proofs in Case 1
(Theorems~\ref{thm:DE1-Sinc-Quad-explicit}
and~\ref{thm:DE1-Sinc-Indef})}

The claims of
Theorems~\ref{thm:DE1-Sinc-Quad-explicit}
and~\ref{thm:DE1-Sinc-Indef} follow
from the next lemma.

\begin{lemma}
Let the assumptions in Theorem~\ref{thm:DE1-Sinc-Quad-explicit}
or Theorem~\ref{thm:DE1-Sinc-Indef}
be fulfilled.
Then, the function $F(\zeta)=f(\DEt{1}(\zeta))\DEtDiv{1}(\zeta)$
belongs to $\LC_{L,R,\alpha,\beta}^{\textDEg}(\domD_d)$
with $L=2^{\nu}K/\{\cos(\frac{\pi}{2}\sin d)\}^{(\nu-\mu)/2}$
and $R=2^{\nu}K$.
\end{lemma}
%\begin{lemma}
%Let the assumptions in Theorem~\ref{thm:DE1-Sinc-Indef}
%be fulfilled.
%Then the function $F(\zeta)=f(\DEt{1}(\zeta))\DEtDiv{1}(\zeta)$
%belongs to $\LC_{L,R,\alpha,\beta}^{\textDEg}(\domD_d)$
%with $L=2^{\nu}K/\{\cos(\frac{\pi}{2}\sin d)\}^{(\nu-\mu)/2}$
%and $R=2^{\nu}K$.
%\end{lemma}
\begin{proof}
First, consider the case $\alpha\leq \beta$.
From the inequality~\eqref{leq:Sinc-quad-case1-alpha}, it follows that
%it follows that
\begin{align*}
|F(\zeta)|&\leq
%\frac{K|\DEtDiv{1}(\zeta)|}{|1+\DEt{1}\zeta|^{(\alpha+1)/2}}
%=
\frac{K}{|1+\rme^{-\pi\sinh\zeta}|^{\alpha/2}|1+\rme^{\pi\sinh\zeta}|^{\beta/2}}
\cdot 2^{\alpha}|1+\rme^{\pi\sinh\zeta}|^{(\beta-\alpha)/2}
%\quad (\Re\zeta<0),
\intertext{for $\zeta\in\domD_d$ with $\Re\zeta<0$,
and from the inequality~\eqref{leq:Sinc-quad-case1-beta}, it
follows that}
|F(\zeta)|&\leq
%\leq\frac{K|\DEtDiv{1}(\zeta)|}{|1+\DEt{1}(\zeta)|^{(\beta+1)/2}}
%=
\frac{K}{|1+\rme^{-\pi\sinh\zeta}|^{\alpha/2}|1+\rme^{\pi\sinh\zeta}|^{\beta/2}}
\cdot
\frac{2^{\beta}}{|1+\rme^{-\pi\sinh\zeta}|^{(\beta-\alpha)/2}}
\end{align*}
for $\zeta\in\domD_d$ with $\Re\zeta\geq 0$.
Setting $\zeta=x+\imnum y$ with $x<0$, we have
\begin{align*}
&2^{\alpha}|1+\rme^{\pi\sinh\zeta}|^{(\beta-\alpha)/2}\\
&=2^{\alpha}(1+\rme^{\pi\sinh(x)\cos y})^{(\beta-\alpha)/2}
\left\{1-\frac{\sin^2(\pi\cosh(x)\sin y)}{\cosh^2(\pi\sinh(x)\cos y)}\right\}^{(\beta-\alpha)/4}\\
&\leq 2^{\alpha}(1+\rme^{0})^{(\beta-\alpha)/2}
\left\{1- 0\right\}^{(\beta-\alpha)/4}=2^{(\alpha+\beta)/2}
\leq 2^{\beta}
\leq \frac{2^{\beta}}{\{\cos(\frac{\pi}{2}\sin y)\}^{(\beta-\alpha)/2}}.
\end{align*}
Furthermore, setting $\zeta=x+\imnum y$ with $x\geq 0$
and using Lemma~\ref{lem:DEfunc-estim}, we have
\begin{align*}
\frac{2^{\beta}}{|1+\rme^{-\pi\sinh\zeta}|^{(\beta-\alpha)/2}}
%&=\frac{2^{\beta}}{(1+\rme^{-2x})^{(\beta-\alpha)/2}
%\left\{1-\frac{\sin^2 y}{\cosh^2 x}\right\}^{(\beta-\alpha)/4}}\\
&\leq
\frac{2^{\beta}}{(1+ \rme^{-\pi\sinh(x)\cos y})^{(\beta-\alpha)/2}
\{\cos(\frac{\pi}{2}\sin y)\}^{(\beta-\alpha)/2}}\\
&\leq \frac{2^{\beta}}{(1+0)\{\cos(\frac{\pi}{2}\sin y)\}^{(\beta-\alpha)/2}}.
\end{align*}
Thus, because $\mu=\alpha$ and $\nu=\beta$ in this case,
it holds for all $\zeta\in\domD_d$ that
\begin{align*}
|F(\zeta)|&\leq
\frac{K}{|1+\rme^{-\pi\sinh\zeta}|^{\alpha/2}|1+\rme^{\pi\sinh\zeta}|^{\beta/2}}
\cdot
\frac{2^{\nu}}{\{\cos(\frac{\pi}{2}\sin d)\}^{(\nu-\mu)/2}},
\intertext{and it holds for all $x\in\mathbb{R}$ that}
%\begin{align*}
|F(x)|&\leq
\frac{K}{(1+\rme^{-\pi\sinh x})^{\alpha/2}(1+\rme^{\pi\sinh x})^{\beta/2}}
\cdot
\frac{2^{\nu}}{\{\cos (\frac{\pi}{2}\sin 0)\}^{(\nu-\mu)/2}}.
\end{align*}
%holds for all $x\in\mathbb{R}$.
In the case $\alpha>\beta$, the same inequalities hold.
This completes the proof.
%Hence the claim follows.
\end{proof}

\subsubsection{Proofs in Case 2
(Theorems~\ref{thm:DE2-Sinc-Quad-explicit}
and~\ref{thm:DE2-Sinc-Indef})}

The claims of
Theorems~\ref{thm:DE2-Sinc-Quad-explicit}
and~\ref{thm:DE2-Sinc-Indef} follow
from the next lemma.

\begin{lemma}
Let the assumptions in Theorem~\ref{thm:DE2-Sinc-Quad-explicit}
or Theorem~\ref{thm:DE2-Sinc-Indef}
be fulfilled.
Then, the function $F(\zeta)=f(\DEt{2}(\zeta))\DEtDiv{2}(\zeta)$
belongs to $\LC_{L,R,\alpha,\beta}^{\textDEg}(\domD_d)$
with $L=K$ and $R=K$.
\end{lemma}
%\begin{lemma}
%Let the assumptions in Theorem~\ref{thm:DE2-Sinc-Indef}
%be fulfilled.
%Then the function $F(\zeta)=f(\DEt{2}(\zeta))\DEtDiv{2}(\zeta)$
%belongs to $\LC_{L,R,\alpha,\beta}^{\textDEg}(\domD_d)$
%with $L=K$ and $R=K$.
%\end{lemma}
\begin{proof}
From the inequality~\eqref{leq:Sinc-quad-case2-alpha-beta} with $i=2$,
\eqref{ineq:LC-DE-complex}
and~\eqref{ineq:LC-DE-real} immediately hold with $L=R=K$.
%it follows that
%\begin{align*}
%|F(\zeta)|\leq
%%K\frac{|\rme^{(\pi/2)\sinh\zeta}|^{\alpha-1}}{|1+\rme^{\pi\sinh \zeta}|^{(\alpha+\beta)/2}}|\DEtDiv{2}(\zeta)|
%K|E_2(\DEt{2}(\zeta);\alpha,\beta)||\DEtDiv{2}(\zeta)|
%=\frac{K(\pi/2)|\cosh\zeta|}{|1+\rme^{-\pi\sinh\zeta}|^{\alpha/2}|1+\rme^{\pi\sinh\zeta}|^{\beta/2}}
%\end{align*}
%for all $\zeta\in\domD_d$
%(which clearly includes $\zeta\in\mathbb{R}$). This completes the proof.
\end{proof}

\subsubsection{Proofs in Case 3
(Theorems~\ref{thm:DE3-Sinc-Quad-explicit}
and~\ref{thm:DE3-Sinc-Indef})}

The claims of
Theorems~\ref{thm:DE3-Sinc-Quad-explicit}
and~\ref{thm:DE3-Sinc-Indef} follow
from the next lemma.

\begin{lemma}
\label{lem:DE3-Sinc-Quad-check}
Let the assumptions in Theorem~\ref{thm:DE3-Sinc-Quad-explicit}
or Theorem~\ref{thm:DE3-Sinc-Indef}
be fulfilled.
Then, the function $F(\zeta)=f(\DEt{3$\ddagger$}(\zeta))\DEtDiv{3$\ddagger$}(\zeta)$
belongs to $\LC_{L,R,2\alpha,2\beta}^{\textDEg}(\domD_d)$
with $L=2(\tilde{c}_{d})^{1-\alpha} K$
and $R=2(\rme^{\pi/12})^{1-\alpha}K$,
where $\tilde{c}_{d}$ is the constant defined in~\eqref{tilde-c-alpha-d}.
\end{lemma}
%\begin{lemma}
%\label{lem:DE3-Sinc-Indef-check}
%Let the assumptions in Theorem~\ref{thm:DE3-Sinc-Indef}
%be fulfilled.
%Then the function $F(\zeta)=f(\DEt{3$\ddagger$}(\zeta))\DEtDiv{3$\ddagger$}(\zeta)$
%belongs to $\LC_{L,R,2\alpha,2\beta}^{\textDEg}(\domD_d)$
%with $L=2(\tilde{c}_{d})^{1-\alpha} K$
%and $R=2(\rme^{\pi/12})^{1-\alpha}K$,
%where $\tilde{c}_{d}$ is a constant defined in~\eqref{tilde-c-alpha-d}.
%\end{lemma}

%In order to prove these two lemmas,
For the proof,
we need the next estimate.

\begin{lemma}
\label{lem:log-exp}
Let $d$ be a constant with $0<d<\pi/2$. Then,
\begin{align}
 \sup_{\zeta\in\overline{\domD_d}}
\left|
\frac{1+\log(1+\rme^{\pi\sinh\zeta})}{\log(1+\rme^{\pi\sinh\zeta})}
\cdot\frac{1}{1+\rme^{-\pi\sinh\zeta}}
\right|
&\leq \tilde{c}_{d},\label{ineq:log-exp-complex}\\
% \frac{1+\log(1+c_d)}{\log(1+c_d)}c_d,
\sup_{x\in\mathbb{R}}\left\{
\frac{1+\log(1+\rme^{\pi\sinh x})}{\log(1+\rme^{\pi\sinh x})}
\cdot\frac{1}{1+\rme^{-\pi\sinh x}}
\right\}
&\leq \rme^{\pi/12},\label{ineq:log-exp-real}
\end{align}
where $\tilde{c}_d$ is a constant defined in~\eqref{tilde-c-alpha-d}.
\end{lemma}
\begin{proof}
First, consider~\eqref{ineq:log-exp-real},
which is proved by showing
\[
p(t)=\frac{1+t}{t}(1-\rme^{-t})\leq 1+\frac{\log 6}{6}
\,\,\,\left(<\rme^{\pi/12}\right)
\]
for $t\geq 0$ (put $t=\log(1+\rme^{\pi\sinh x})$).
Let $\lambda$ be a value with $\log 6<\lambda<\log 7$
so that $p'(\lambda)=(1+\lambda+\lambda^2-\rme^{\lambda})/(\rme^{\lambda}\lambda^2)=0$.
Then, $p(t)\leq p(\lambda)$ clearly holds.
Furthermore, using the relation
$1+\lambda+\lambda^2=\rme^{\lambda}$, we have
\[
 p(\lambda)
=\frac{1+\lambda}{\lambda}\left(\frac{\rme^{\lambda}-1}{\rme^{\lambda}}\right)
=\frac{1+\lambda}{\lambda}\left(\frac{\lambda(1+\lambda)}{\rme^{\lambda}}\right)
=\frac{(1+\lambda+\lambda^2)+\lambda}{\rme^{\lambda}}
=1+\lambda\rme^{-\lambda}.
\]
Because the function $1+x\rme^{-x}$ is monotonically decreasing
for $x\geq 1$,
$1+\lambda\rme^{-\lambda}< 1+(\log 6)\rme^{-\log 6}$ holds
(note that $\log 6<\lambda$).
Hence,~\eqref{ineq:log-exp-real} is proved.

Next, consider~\eqref{ineq:log-exp-complex}.
By the maximum modulus principle, it is proved by showing
\begin{equation}
|p(\xi)|\leq \frac{1+\log(2+\rme^{\gamma})}{\log(2+\rme^{\gamma})}
(1+\rme^{\gamma})=\tilde{c}_d,
\label{leq:target-log-exp-complex}
\end{equation}
where $\xi=\log(1+\rme^{\pi\sinh(x+\imnum d)})$
and $\gamma=-\log[\cos(\frac{\pi}{2}\sin d)]$.
Here, notice that
\[
 \Re\xi
=\log|1+\rme^{\pi\sinh(x+\imnum d)}|
\geq \log\left[
(1+\rme^{\pi\sinh(x)\cos d})\cos(\textstyle\frac{\pi}{2}\sin d)
\right]
\geq -\gamma
\]
holds from Lemma~\ref{lem:DEfunc-estim}.
Let us bound $|p(\xi)|$
in the two cases: a) $|\xi|\leq\log(2+\rme^{\gamma})$
and b) $|\xi|>\log(2+\rme^{\gamma})$.
For case a), it holds that
\begin{align*}
|p(\xi)|
&=\left|(1+\xi)\sum_{k=1}^{\infty}\frac{(-\xi)^{k-1}}{k!}\right|
\leq (1+|\xi|)\sum_{k=1}^{\infty}\frac{|\xi|^{k-1}}{k!}
=\frac{1+|\xi|}{|\xi|}(\rme^{|\xi|}-1).
%\leq
%\frac{1+\gamma}{\gamma}(\rme^{\gamma}-1)
\end{align*}
Furthermore, because $(1+x)(\rme^{x}-1)/x$ is monotonically increasing,
we have~\eqref{leq:target-log-exp-complex}.
%\[
% \frac{1+|\xi|}{|\xi|}(\rme^{|\xi|}-1)
%\leq
%\frac{1+\log(2+\rme^{\gamma})}{\log(2+\rme^{\gamma})}
%(1+\rme^{\gamma})=\tilde{c}_d.
%\]
For case b), because $\Re\xi\geq\gamma$, it holds that
\[
 |p(\xi)|
\leq\frac{1+|\xi|}{|\xi|}
(1+|\rme^{-\xi}|)
=\frac{1+|\xi|}{|\xi|}
(1+\rme^{-\Re\xi})
\leq\frac{1+|\xi|}{|\xi|}
(1+\rme^{\gamma}).
%\leq
%\frac{1+\log(2+\rme^{\gamma})}{\log(2+\rme^{\gamma})}
%(1+\rme^{\gamma})
%=\tilde{c}_d.
\]
Furthermore, because $(1+x)/x$ is monotonically decreasing,
we have~\eqref{leq:target-log-exp-complex}.
%which is proved by showing
%\[
%\frac{1+|\xi|}{|\xi|}
%(1+\rme^{\gamma})
%\leq \frac{1+\log(2+\rme^{\gamma})}{\log(2+\rme^{\gamma})}
%(1+\rme^{\gamma})=\tilde{c}_d.
%\]
This completes the proof.
\end{proof}

Lemma~\ref{lem:DE3-Sinc-Quad-check}
%and~\ref{lem:DE3-Sinc-Indef-check}
is proved as follows.

\begin{proof}
From the inequality~\eqref{leq:Sinc-quad-case2-alpha-beta} with $i=3$,
it follows that
\begin{align*}
|F(\zeta)|
&\leq \left|\frac{1+\log(1+\rme^{\pi\sinh\zeta})}{\log(1+\rme^{\pi\sinh\zeta})}\right|^{1-\alpha}
\frac{|\DEtDiv{3$\ddagger$}(\zeta)|}{|1+\rme^{\pi\sinh\zeta}|^{\beta}}\\
&=2\left|
\frac{1+\log(1+\rme^{\pi\sinh\zeta})}{\log(1+\rme^{\pi\sinh\zeta})}
\cdot\frac{1}{1+\rme^{-\pi\sinh\zeta}}
\right|^{1-\alpha}
\frac{K(\pi/2)|\cosh\zeta|}{|1+\rme^{-\pi\sinh\zeta}|^{\alpha}|1+\rme^{\pi\sinh\zeta}|^{\beta}}.
\end{align*}
Then, use Lemma~\ref{lem:log-exp} to obtain the desired result.
\end{proof}

%\section{Including Figures and Other Graphics}
%The correct way to include figures is with the
%\texttt{graphicx} package.  An example is
%Be sure to include all graphics files in any final submissions.
%
%\noindent{}Note: The current preferred graphics compilation scheme is
%with PDF-\LaTeX, but we can also accept encapsulated Postscript ($*$.eps) figures
%or figures in other formats if necessary.
%
%
%\section{Managing Citations}
%
%Use of \BibTeX{} is encouraged, since is automates formatting
%the bibliographic entries properly, as well as numbering and
%referencing them.  An example bibliography, used for creating
%the reference in this document, is included at
%\url{http://interval.louisiana.edu/reliable-computing-journal/example.bib},
%while an extensive collection of \BibTeX{} bibliographies can
%be found at
%\url{http://www.math.utah.edu/~beebe/bibliographies.html}.  We
%supply \BibTeX{} entries to articles appearing in this journal,
%as do many other journals and citation clearinghouses.
%
%\section{Style Notes} \label{sec:style-notes}
%In \textit{Reliable Computing's} style, words in the title, section headings,
%and subsection headings, other than articles and conjunctions, are 
%capitalized.
%

%\section*{Acknowledgements}
%This work was supported by JSPS
%Grant-in-Aid for Young Scientists (B)(24760060).

\begin{small}
\bibliographystyle{myamsplain}
\bibliography{2012scan-error-estim.bib}
\end{small}
\end{document}